\documentclass[11pt,a4paper]{article}

\newcommand{\remove}[1]{#1}

\usepackage[IL2]{fontenc}
\usepackage{chngcntr, apptools}
\AtAppendix{\counterwithin{theorem}{section}}
\usepackage{amsmath,amssymb,graphicx, amsthm, soul}
\usepackage[heads=LaTeX]{diagrams}
\usepackage{url}
\usepackage[dvipsnames]{xcolor}
\usepackage{graphicx}
\usepackage{multicol}
\usepackage[all]{xy}
\SelectTips{cm}{10}

\usepackage{wrapfig}
\newtheorem{theorem}{Theorem} 

\newtheorem{proposition}{Proposition}[section]
\newtheorem{observation}[proposition]{Observation}
\newtheorem{lemma}[proposition]{Lemma}

\newtheorem{problem}[proposition]{Problem}
\theoremstyle{definition}

\newtheorem{remark}[proposition]{Remark}
\theoremstyle{plain}

\newcommand{\Long}[1]{} 
\usepackage{ifthen}
\usepackage{geometry}
\newboolean{MareksFormat}
\InputIfFileExists{MareksSwitch}{}{}
\ifthenelse{\boolean{MareksFormat}}{
\geometry{papersize={145mm,180mm},top=2.5mm,bottom=2.5mm,left=1.5mm,right=1.5mm}
}{
}

\DeclareMathOperator{\id}{id}
\DeclareMathOperator{\incl}{incl}
\DeclareMathOperator{\const}{const}

\newcommand{\s}{\sigma}
\renewcommand{\t}{\tau}

\DeclareMathOperator{\Cone}{Cone}
\newcommand{\argmin}{\mathrm{arg\, min}}

\def\:{\colon}
\def\Hopf{{\eta}}
\newcommand{\R}{\mathbb{R}}

\newcommand{\N}{\mathbb{N}}
\newcommand{\Z}{\mathbb{Z}}

\newcommand{\heading}[1]{\vspace{1ex}\par\noindent{\bf\boldmath #1}}
\def\indef#1{\emph{#1}}
\def\Zf{Z_r(f)}
\renewcommand{\Im}{\mathrm{Im}}

\usepackage{soul}

\newcommand{\zavr}{]}

\usepackage{microtype}

\title{{On Computability and Triviality of Well Groups\footnote{This is an extended version of a paper that is to appear in the proceedings
of the Symposium on Computation Geometry 2015.} \footnote{This
research was supported 
by institutional support RVO:67985807 and by the People Programme (Marie Curie Actions) of the European Union's Seventh Framework Programme (FP7/2007-2013) under REA grant agreement n° [291734].}}
}

\author{Peter Franek, Marek Kr\v{c}\'al}

\begin{document}

{
%
%

\maketitle

\thispagestyle{empty}
\begin{abstract}
The concept of \emph{well group} in a special but important case captures homological properties of the zero set of a continuous map $f\:K\to \R^n$ on a compact space $K$ that are invariant with respect to perturbations of $f$. The perturbations are arbitrary continuous maps within $L_\infty$ distance $r$ from $f$ for a given $r>0$. 
The main drawback of the approach is that the computability of well groups was shown only when $\dim K=n$ or $n=1$. 

Our contribution to the theory of well groups is twofold: on the one hand we improve on the computability issue, but on the other hand we present a range of examples where the well groups are incomplete invariants, that is, fail to capture certain important robust properties of the zero set.

For the first part, we identify a computable subgroup of the well group that is obtained by cap product with the pullback of the orientation of $\R^n$ by $f$.
In other words, well groups can be algorithmically approximated from below. When $f$ is smooth and $\dim K<2n-2$, our approximation of  the $(\dim\ K-n)$th well group is exact.

For the second part, we find examples of maps $f,f'\:K\to \R^n$ with all well groups isomorphic but whose perturbations have different zero sets. We discuss on a possible replacement of the well groups of vector valued maps by an invariant of a better descriptive power and computability status.
\end{abstract}
}  
\newpage

\section{Introduction}
\label{s:intro}

In many engineering and scientific solutions, a highly desired property is the resistance against noise or perturbations.
We can only name a fraction of the instances: stability in data analysis~\cite{Carlsson:2009}, robust optimization~\cite{Ben:2009}, image processing~\cite{Goudail:2004}, or stability of numerical methods~\cite{Higham:2002}.
Some very important tools for robust design come from topology, which can capture stable properties of spaces and maps.

In this paper, we take the robustness perspective on the study of the solution set of systems of nonlinear equations, a fundamental problem 
in mathematics and computer science. 
Equations arising in mathematical modeling of real problems are usually inferred from observations, measurements or previous computations. 
We want to extract maximal information about the solution set, if an estimate of the error in the input data is given. 

More formally, for a continuous map \(f:K\to \R^n\) on a compact Hausdorff space $K$ and $r>0$ we want to study properties of the family of zero sets \begin{displaymath}
Z_r(f):=\{g^{-1}(0)\:\|f-g\|\le r\},
\end{displaymath}
where \(\|\cdot\|\) is the max-norm with respect to some fixed norm $|\cdot|$ in \(\R^n\). 
The functions $g$ with \(\|f-g\|\leq r\) (or $\|f-g\|<r$) will be referred to as \indef{$r$-perturbations of $f$} (or \indef{strict $r$-perturbations of $f$}, respectively). 
Quite notably, we are not restricted to \emph{parameterized} perturbations but allow arbitrary continuous functions at most (or less than) $r$ far from  $f$ in the max-norm.

\heading{Well groups.} Recently, the concept of well groups was developed to measure ``robustness of
intersection'' of a map $f\:K\to Y$ with a subspace $Y'\subseteq Y$~\cite{well-group}.

In the special but very important case when $Y=\R^n$ and $Y'=\{0\}$ it is a property of $\Zf$ that, informally speaking, captures ``homological properties'' that are common to all zero sets in $\Zf$.  We enhance the theory to include a~\emph{relative    case}\footnote{Authors of \cite{interlevel} develop a different notion of relativity that is based on considering a pair of spaces $(Y',Y'_0)$ instead of the single space $Y'$. This direction is rather orthogonal to the matters of this paper.} that is especially convenient in the case when $K$ is a manifold with boundary.
Let $B\subseteq K$ be a pair of compact Hausdorff spaces and $f: K\to\R^n$ continuous.  
Let 
$X:=|f|^{-1} [0,r]$ where $|f|$ denotes the function
$x\mapsto |f(x)|$; this is the smallest space containing zero sets of all $r$-perturbations $g$ of $f$. In the rest of the paper, for any space $Y\subseteq K$ we will abbreviate the pair \((Y,Y\cap B\)) by \((Y,B)\) and, similarly for homology, \(H_*(Y,Y\cap B)\)) by \(H_*(Y,B)\). Everywhere in the paper we use homology and cohomology groups with coefficients in $\Z$ unless explicitly stated otherwise. For brevity we omit the coefficients from the notation. 

The $j$th well group $U_j(f,r)$ of $f$ at radius $r$ is the subgroup of $H_j(X,B)$ defined by
$$
U_j(f,r):=\bigcap_{Z\in Z_r(f)} \, \mathrm{Im}\big (H_j(Z, B)\stackrel{i_*}{\longrightarrow} H_j(X,B)\big),
$$
where $i_*$ is induced by the inclusion $i\:g^{-1}(0)\hookrightarrow X$ and $H$ refers to a convenient homology theory of compact metrizable spaces that we describe
below.\footnote{In \cite{well-group,interlevel}, well groups were defined by means of singular homology. But then, once we allow arbitrary continuous perturbations, to the best of our knowledge, no $f\:K\to\R^n$ with nontrivial $U_j(f,r)$ for $j>0$ would be known. In particular, the main result of \cite{interlevel} would not hold. The correction via means of Steenrod homology was independently identified by the authors of \cite{interlevel}.} For a simple example of a map $f$ with $U_1(f,r)$ nontrivial see Figure~\ref{f:well}.

\heading{Significance of well groups.} We only mention a few of many interesting things mostly related to our setting. The well group in dimension zero characterizes robustness of solutions of a system of equations $f(x)=0$. Namely, $\emptyset\in Z_r(f)$ if and only if $U_0(f,r)\cong 0$. Higher well groups capture additional robust topological properties of the zero set such as in Figure~\ref{f:well}. Perhaps the most important is their ability to form \emph{well diagrams} \cite{well-group}---a kind of measure for robustness of the zero set (or more generally, robustness of the intersection of $f$ with other subspace $Y'\subseteq Y)$. The well diagrams are stable with respect to taking perturbations of $f$.\footnote{Namely, so called \emph{bottleneck distance} between a well diagrams of $f$ and $f'$ is bounded by $\|f-f'\|$. The stability does not say how well the well diagrams describe the zero set. This question is also addressed in this paper.} 
\heading{Homology theory.} For the foundation of well groups we need a homology theory on compact Hausdorff spaces
that satisfies some additional properties that we specify later in Section~\ref{s:cap}. Roughly
speaking, we want that the homology theory behaves well with
respect to infinite intersections. Without these properties we would have to consider only ``well behaved'' perturbations of a given $f$ in order to be able to obtain some nontrivial well groups in dimension greater than zero. We explain this in more detail also in Section~\ref{s:cap}.
For the moment it is enough to say that  the \emph{\v Cech} homology can be used and that for any computational purposes it behaves 
like simplicial homology. 
In Section~\ref{s:cap} we explain why using singular homology would make the notion of well groups trivial. 

A basic ingredient of our methods is the notion of \emph{cap product} 
$$
\frown: H^n(X,A)\otimes H_k(X,A\cup B)\to H_{k-n}(X,B)
$$
between cohomology and homology. 
We refer the reader to~\cite[Section 2.2]{prasolov} and~\cite[p. 239]{Hatcher} for its properties
and to Appendix~\ref{a:Cech} for its construction in \v Cech (co)homology. 
Again, it behaves like the simplicial cap product when applied to simplicial complexes. For an algorithmic implementation, one 
can use its simplicial definition from~\cite{prasolov}.

\subsection{Computability results}    
\begin{figure}
\begin{center}
\includegraphics[scale=1.1]{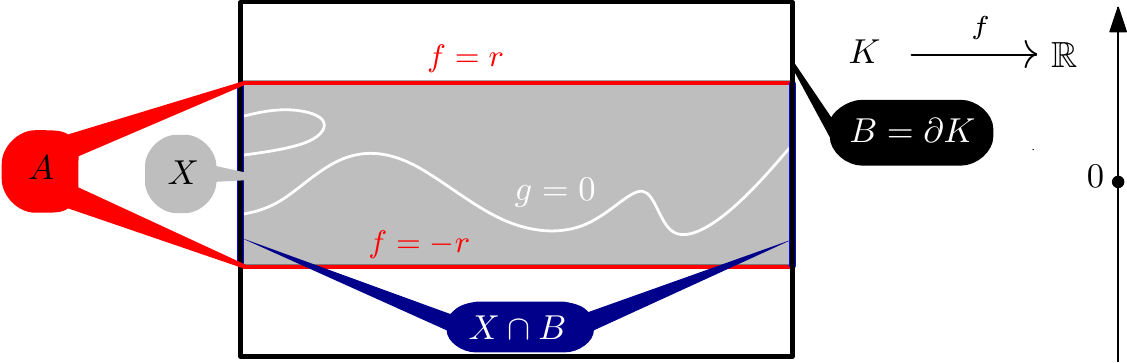}
\end{center}
\caption{For the projection $f(x,y)=y$ to the vertical axis defined on a box $K$,
the zero set of every $r$-perturbation is contained in $X=|f|^{-1}[0,r]$ and $\partial X$ consists of $A$ (upper and lower 
side) where $|f|=r$, and $X\cap B\subseteq\partial K$. The zero set always separates the two components of $A$. On the homological level, the zero set ``connects'' the two components of $X\cap B$ and the image of $H_1(g^{-1}(0),B)$ in $H_1(X,B)$ is always surjective and thus $U_1(f,r)\cong H_1(X,B)$. Note that the well group would be trivial with $B=\emptyset$.}
\label{f:well}
\end{figure}


\heading{Computer representation.}
To speak about computability, we need to fix some computer representation of the input. Here we assume the simple but general setting of~\cite{nondec}, namely, $K$ is a finite simplicial
complex, $B\subseteq K$ a subcomplex, $f$ is simplexwise linear with rational values on vertices\footnote{We emphasize that the considered $r$-perturbations of $f$ need not be neither simplexwise linear nor have rational values on the vertices.} and
the norm $|\cdot|$ in $\R^n$ can be (but is not restricted to) $\ell_1,\ell_2$
or $\ell_\infty$ norm. 

\heading{Previous results.} The algorithm for the computation  of well groups was developed only  in the particular cases
of $n=1$~\cite{interlevel} or $\dim K=n$~\cite{chazal}. 
In \cite{nondec} we  settled the computational complexity of the well group $U_0(f,r)$. The complexity is essentially identical to deciding whether the restriction $f|_A\:A\to S^{n-1}$ can be extended to $X\to S^{n-1}$ for $A=|f|^{-1}(r)$, or equivalently, $A=f^{-1}(S^{n-1})$. 
The extendability problem can be decided as long as $\dim K\leq 2n-3$ or $n=1,2$ or $n$ is even. On the contrary, the extendability of maps into a sphere---as well as triviality of $U_0(f,r)$---cannot be decided for $\dim K\ge 2n-2$ and $n$ odd, see \cite{nondec}.\footnote{We cannot even approximate the ``robustness of roots'': it is undecidable, given a~simplicial complex $K$ and a simplexwise linear map $f\:K\to\R^n$, whether there exists $\epsilon>0$ such that $U_0(f,\epsilon)$ is nontrivial or whether $U_0(f,1)$ is trivial. The extendability can always be decided for $n$ even, however, the problem is less likely tractable for $\dim K>2n-2$.} 
In this paper we shift our attention to higher well groups. 

\heading{Higher well groups---extendability revisited.}  
The main idea of our study of well groups is based on the following.
We try to find $r$-perturbations of $f$ with as small zero set
as possible, that is, avoiding zero on $X'$ for $X'\subseteq
X$ as large as possible. 
We will show in Lemma~\ref{l:perturb-ext} that for
each strict 
$r$-perturbation $g$ of $f$ we can find an extension $e\:X\to\R^n$
of $f|_A$ with $g^{-1}(0)=e^{-1}(0)$ and vice versa. Thus
equivalently, we try to extend $f|_A$  to a map $X'\to S^{n-1}$ for $X'$ as
large as possible. The higher skeleton\footnote{The $i$-skeleton $X^{(i)}$ of a simplicial (cell) complex $X$ is the subspace of $X$ containing all simplices (cells) of dimension at most $i$.} of $X$ we cover, the more well groups we kill.
\begin{observation}\label{o:triv}
Let $f\:K\to \R^n$ be a map on a compact space. Assume that the pair of spaces $A\subseteq X$ defined as $|f|^{-1}(r)\subseteq |f|^{-1}[0,r]$, respectively, can be triangulated and $\dim X=m$. If the map $f|_A$ can be extended to a map $A\cup X^ {(i-1)} \to S^{n-1}$ then $U_j(f,r)$ is trivial for $j>m-i$.
\end{observation}

Assume, in addition, that there is no extension $A\cup X^{(i)}\to S^{n-1} $. By the connectivity of the sphere  $S^{n-1}$, we have $i\ge n$. Does the lack of extendability to $X^{(i)}$ relate to higher well groups, especially $U_{m-i}(f,r)$? The answer is \emph{yes} when $i=n$ as we show in our computability results below. On the other hand, when $i>n$, the lack of extendability \emph{is not} necessarily reflected by $U_{m-i}(f,r)$. This leads to the incompleteness results we show in the second part of the paper.

\heading{The first obstruction.} The lack of extendability of $f|_A$ to the $n$-skeleton is measured by the so called \emph{first obstruction} that is defined in terms of cohomology theory as follows. We can view $f$ as a map of pairs $(X,A)\to (B^n, S^{n-1})$ where
$B^n$ is the ball bounded by the sphere $S^{n-1}:=\{x\:|x|= r\}$.
Then the first obstruction $\phi_f$ is equal to the pullback $ f^*(\xi)\in H^n(X,A)$ of the fundamental cohomology class $\xi^n\in H^n(B^n,S^{n-1})$.
\footnote{This is the global description of the first obstruction as presented in \cite{whitehead-obstructionTheory}. It can be shown that $\phi_f$ depends on the homotopy class of $f|_A$ only.
Another way of defining the first obstruction is the following. It is represented by the so-called \emph{obstruction cocycle} $z_f\in Z^n(X,A)$ that assigns to each $n$-simplex $\s\in X$ the element $[f|_{\partial\s}]\in\pi_{n-1}(S^{n-1}) \cong \Z$~\cite[Chap. 3]{prasolov}. Through this definition it is not difficult to derive that the map $f|_A$ can be extended to $X^{(n)}\to S^{n-1}$ if and only if $\phi_f=0$, see also~\cite[Chap. 3]{prasolov}.}

\begin{theorem}
\label{t:cap}
Let $B\subseteq K$ be compact spaces and let $f: K\to\R^n$ be continuous.
Let $|f|^{-1}[0,r]$ and $ |f|^{-1}(r)$ be denoted by $X$ and $A$, respectively, and   $\phi_f$ be the first obstruction.
Then $\phi_f\frown H_k(X,A\cup B)$ is a subgroup of $U_{k-n}(f,r)$ for each $k\geq n$.
\end{theorem}
Our assumptions on computer representation allow for simplicial approximation of $X,A$ and $f$. The pullback of $\xi^n\in H^n(B^n,S^{n-1})$ and the cap product can be computed by the standard formulas. This together with more details worked out in the proof in Section~\ref{s:cap} gives the following.
\begin{theorem}\label{o:comp}
Under the assumption on computer representation
of $K, B$ and $f$ as above, the subgroup $\phi_f\frown H_k(X,A\cup B)$ of $U_{k-n}(f,r)$ (as in Theorem~\ref{t:cap}) can be computed. \end{theorem}

\heading{The gap between $U_{k-n}$ and $\phi_f\frown H_k(X,A\cup B)$.} There are maps $f$ with $\phi_f$ trivial but nontrivial $U_0(f,r)$.\footnote{This is the case for $f\:\R^4\to\R^3$ given by $f(x):=|x|\eta(x/|x|)$ where $\eta\:S^3\to S^2$ is the Hopf map.} But this can be detected by the above mentioned extendability criterion. We do not present an example
where $U_{k-n}(f,r)\neq\phi_f\frown H_k(X,A\cup B)$ for $k-n>0$, although the inequality is possible in general.
In the rest of the paper we work in the other direction to show that there is no gap in various cases and various dimensions.

An important instance of Theorem~\ref{t:cap} is the case when $X$ can be equipped with the structure of a smooth orientable manifold. 

\begin{theorem}
\label{t:Poincare}
Let $f\:K\to \R^n$ and $X,A$ be as above. Assume that $X$ can be equipped with a smooth orientable manifold structure, $A=\partial X$, $B=\emptyset$ and $n+1\leq m\leq 2n-3$ for $m=\dim X$. Then
$$
U_{m-n}(f,r)=\phi_f\frown H_m(X,\partial X).
$$
\end{theorem}
When $m=n$, the well group $U_0(f,r)$ can be strictly larger than $\phi_f\frown H_n(X,\partial X)$ but it can be computed.

We believe that the same claim holds when $X$ is an orientable PL manifold. It remains open whether the last equation holds also for $m> 2n-3$.
Throughout the proof of Theorem~\ref{t:Poincare}, we will show that if $g: K\to\R^n$ is a~smooth $r$-perturbation of $f$ transverse to $0$, then
the fundamental class of $g^{-1}(0)$ is mapped  to the Poincar\'e dual of the first obstruction. This also holds if $B\neq\emptyset$
and in all dimensions.

\subsection{Well groups $U_*(f,r)$ are incomplete as an invariant of $Z_r(f)$.} 
A simple example illustrating Theorem~\ref{t:Poincare} is the map $f\:S^2\times B^3\to\R^3$ defined by $f(x,y):=y$ with $B^3$ considered as the unit ball in $\R^3$. It is easy to show that 
\begin{equation}\label{e:prop}
\text{for every }1\text{-perturbation }g \text{ of }f\text{ and every }x\in S^2\text{ there is a root of }g\text{ in }\{x\}\times B^3.\end{equation} 
This robust property is nicely captured by (and can be also derived from) the fact $U_2(f,1) \cong \Z$.

The main question of Section~\ref{s:incompleteness} is what happens, when the first obstruction $\phi_f$ is trivial---and thus $f|_A$ can be extended to $X^{(n)}$---but the map $f|_A$ cannot be extended to  whole of $X$. The zero set of $f$ can still have various robust properties such as \eqref{e:prop}. It is the case of $f\:S^2\times B^4\to \R^3$ defined by $f(x,y):=|y|\Hopf(y/|y|)$ where $\Hopf\:S^3\to S^2$ is a homotopically nontrivial map such as the Hopf map. The zero set of each $r$-perturbation $g$ of $f$ intersects each section $\{x\}\times B_4$, but unlike in the example before, well groups do not capture this property.
All well groups $U_j(f,r)$ are trivial for $j>0$ and,\footnote{Namely $U_2(f,r)\cong 0$ as is shown by the $r$-perturbation $g(x,y)=f(x,y)-rx$ with the zero set homeomorphic to the $3$-sphere.} consequently, they
cannot distinguish $f$ from another $f'$ having only a single robust root in $X$.
We will describe the construction of such $f'$ for a wider range examples.

 In the following, $B^i_q$ will denote the $i$-dimensional ball of radius $q$, that is, $B^i_q=\{y\in\R^i\:|y|\le q\}$.
We also emphasize that this and the following theorem hold for arbitrary coefficient group of the homology theory $H_*$.
\begin{theorem}\label{t:torus}
Let $i,m, n\in\N$ be such that $m-i<n<i< (m+n+1) /2$ and both $\pi_{i-1} (S^{n-1})$ and $\pi_{m-1}(S^{n-1})$ are nontrivial. Then on $K=S^{m-i}\times B^i_1$ we can define two maps $f,f'\:K\to \R^n$ such that for all $r\in(0,1]$
\begin{itemize}
\item $f$, $f'$ induce the same $X= S^{m-i}\times B^i_r$ and $A=\partial X$ and have the same well groups for any coefficient group of the homology theory $H_*$ defining the well groups, 
\item but $Z_r(f)\neq Z_r(f')$.
\end{itemize}
In particular, the property
$$ \text{for each } Z\in Z_r(.)\text{ and } x\in S^{m-i}\ \text{ there exists } y\in
B^i_r\text{ such that } (x,y)\in Z$$
is satisfied for $f$ but not for $f'$. Namely, $Z_\epsilon(f')$ contains a singleton for each $\epsilon>0$.
\end{theorem}

\remove{In Section~\ref{s:incompleteness} we discuss that the maps $f$ and $f'$ are no peculiar examples but rather typical choices given that the underlying space $K$ is the solid torus $S^{m-i}\times B^i$ and that both $Z_r(\cdot)$ are nontrivial. 
Further we indicate that the same result holds for even more realistic choice of the underlying space $K=B^m$ and $B=\partial K$. For the sake of exposition, we chose the case where $f$ is large on the boundary of $K$ and we do not need to consider nonempty $B$.}

\heading{The lack of extendability not reflected by $U_{m-i}(f,r)$.} 
The key property of the example of Theorem~\ref{t:torus} is that the maps $f|_A$ and $f'|_A$ can be extended to the $(i-1)$-skeleton $X^{(i-1)}$ of $X$, for $i>n$. The difference between the maps lies in the extendability to $X^{(i)}$. Unlike in the case when $i=n$, the lack of extendability is not reflected by the well groups. The crucial part is the triviality of the well groups in dimension $m-i$ and\footnote{This dimension is somewhat important as all higher well groups are trivial by Lemma~\ref{l:triv} and all lower homology groups of $X$ may be trivial as is the case in Theorem~\ref{t:torus}. On the other hand, $H_{m-i}\big(X,\pi_{i-1}(S^{n-1})\big)$ has to be nontrivial in the case when $X$ is a manifold for the reasons following from obstruction theory and Poincar\'e duality.} this triviality holds in  greater generality:

\begin{theorem}\label{t:triviality}
Let $f\:K\to\R^n$, $B\subseteq K$, $X:=|f|^{-1}[0,r]$ and $A:=|f|^{-1}\{r\}$. Assume that the pair  $(X,A)$ can be finitely triangulated.\footnote{ That is, there exist finite simplicial complexes
$A^\Delta\subseteq X^\Delta$ and a homeomorphism
$(X^\Delta, A^\Delta
)\to (X,A).$} 
Further assume that $f|_A$ can be extended to a map $h: A\cup X^{(i-1)}\to S^{n-1}$ for some $i$ such that $m-i<n<i< (m+n)/2$ for $m:=\dim X$. Then $U_{m-i}(f,r)=0$ for any coefficient group of the homology theory $H_*$.
\end{theorem}
The proof is all delegated
to Appendix~\ref{a:triv} as its core idea is already contained in the proof of Theorem~\ref{t:torus}. There we also comment on the possibility of finding pairs of maps $f$ and $f'$ with the same well groups but different robust properties of their zero sets in this more general situation.

\remove{One could ask the question of triviality in dimensions smaller than $m-i$ as well. Our favorite problem is the following
one.\begin{problem} Let $f$ be as in Theorem~\ref{t:triviality} and let $i=n+1$, that is, the first obstruction is trivial. Is it true that all well groups $U_{j}(f,r)$ for $j\ge(m-n+2)/2$ are trivial?
\end{problem}
The bound $j\ge (m-n+2)/2$ is not known to be necessary (we only know that the statement is not true for $j=1$). But passing the bound seems to bring various technical difficulties such as inapplicability of the Freudenthal suspension theorem.}

\heading{Our subjective judgment on well groups of~$\R^n$-valued maps.}
We find the problem of the computability of well groups interesting and challenging with connections to homotopy theory (see also Proposition~\ref{p:characterization} below). 
Moreover, we acknowledge that well groups may be accessible for non-topologists: they are based on the language of homology theory that is relatively intuitive and easy to understand. 
On the other hand, well groups may not have sufficient descriptive power for various situations (Theorems~\ref{t:torus} and~\ref{t:triviality}).
Furthermore, despite all the effort, the computability of well groups seems far from being solved.
In the following paragraphs, we propose an~alternative based on homotopy and obstruction theory that addresses these drawbacks.

\subsection{Related work}
\heading{A replacement of well groups of $\R^n$-valued maps.} In a companion paper~\cite{cobordism}, we find a complete invariant for an enriched version of $Z_r(f)$. The starting point is the surprising claim that $Z_r(f)$---an object of a geometric nature---is determined by terms of homotopy theory.
 
\begin{proposition}[{\cite
{cobordism}}]\label{p:characterization}
Let $f\:K\to \R^n$ be a continuous map on a compact Hausdorff domain, $r>0$, and let  us denote the
space $|f|^{-1}[r,\infty]$ by $A_r$. Then the set
$Z_r(f):=\{g^{-1}(0)\: \|g-f\|\le r\}$ is  determined by the
pair $(K,A_r)$ and the
homotopy class of $f|_{A_r}$ in $[A_r,\{x\in \R^n\:\|x\|\ge r\}]\cong[A_r,S^{n-1}]$.\footnote{Here $[A_r,S^{n-1}]$
denotes the set of all homotopy classes of maps from $A_r$ to $S^{n-1}$, that is, the cohomotopy group $\pi^{n-1}(A_r)$ when $\dim A_r\le 2n-4$.}
\end{proposition}
Since \cite{cobordism} has not been published yet, 
we append the complete proof of Proposition~\ref{p:characterization} in Appendix~\ref{a:characterization}.

Note that since the well groups is a property of $Z_r(f)$, they are determined
by the pair $(K,A_r)$ and the homotopy class $[f|_{A_r}]$. Thus the homotopy class has a greater descriptive power and the examples from the previous section show that this inequality is strict. If $K$ is a simplicial complex, $f$ is simplexwise linear and $\dim\ A_r\le 2n-4$ then $[A_r,S^{n-1}]$ has a natural structure of an Abelian group  denoted by $\pi^{n-1}(A_r).$ The restriction $\dim A_r\le 2n-4$ does not  apply when $n=1,2$ and\footnote{Note
that for $n=1$ the structure of the set $[A,S^{n-1}]$ is very
simple and for $n=2$ we have $[A,S^{n-1}]\cong H^1(A;\Z)$ no
matter what the dimension of $A_r$ is.} otherwise 
we could replace $[A_r,S^{n-1}]$ with $[A_r^{(2n-4)}, S^{n-1}]$ which contains less information but is computable.
The isomorphism type of $\pi^{n-1}(A_r)$ together with the distinguished element
$[f|_{A_r}]$ can be computed essentially by~\cite[Thm 1.1]{post}. Moreover, the inclusions $A_s\subseteq A_r$ for $s\ge r$ induce computable homomorphisms between the corresponding pointed Abelian groups. Thus for a given $f$ we obtain a sequence of pointed Abelian groups $\pi^{n-1}(A_r), {r>0}$ and it can be easily shown that the interleaving distance of the sequences $\pi^{n-1}\big(A_*(f)\big)$ and $\pi^{n-1}\big(A_*(g)\big)$ is bounded by $\|g-f\|$. Thus after tensoring the groups by an arbitrary field, we get persistence diagrams (with a distinguished bar) that will be stable with respect to  the bottleneck distance and the $L_\infty$ norm. The construction will be detailed in~\cite{cobordism}.

The computation of the cohomotopy group $\pi^{n-1}(A)$ is naturally segmented into a hierarchy of approximations of growing computational complexity. Therefore our proposal allows for compromise between the running time and the descriptive power
of the outcome.
The first level of this hierarchy is the primary obstruction $\phi_f$. One could form similar modules of cohomology groups with a distinguished element as we did with the cohomotopy groups above. However, in this paper we passed to homology via cap product in order to relate it to the established well groups. In the ``generic'' case when $X$ is a manifold no information is lost as from the Poincar\'e dual $\phi_f \frown [X]$ we can reconstruct the primary obstruction $\phi_f$ back.

\heading{The cap-image groups.} 
The groups $\phi_f\frown H_k(X,A)$ (with $B=\emptyset$) has been
studied by Amit K. Patel under the name \emph{cap-image groups.} In fact, his setting is slightly more complex with $\R^n$
replaced by arbitrary manifold $Y$. Instead of the
zero sets, he considers preimages of all points of $Y$ simultaneously
in some sense. Although his ideas have not been published yet,
they influenced our research; the application of the cap product
in the context of well groups should be attributed to Patel.\footnote{We
originally proved that when $K$ is a triangulated orientable
manifold, the Poincar\'e dual of $\phi_f$ is contained in $U_{m-n}(f,r)$.
Expanding the proof was not difficult, but the preceding inspiration
of replacing the Poincar\'e duality by cap product came from
Patel. The cap product provides a nice generalization to an arbitrary
simplicial complex $K$.}

The advantage of the primary obstructions and the cap image groups is their computability and well understood mathematical structure. 
However, the incompleteness results of this paper apply to these invariants as well.

\heading{Verification of zeros.} An important topic in the interval
computation community is the verification of the (non)existence
of zeros of a given function~\cite{Neumaier:90}. While the nonexistence
can be often verified by interval arithmetic alone,
a proof of existence requires additional methods which often
include topological considerations. In the case of continuous
maps $f: B^n\to\R^n$, Miranda's or Borsuk's theorem can be used for zero verification~\cite{Frommer:05,Alefeld:01a},
or the computation of the topological degree~\cite{Kearfott:02,Collins:08b,Franek_Ratschan:2012}.
Fulfilled assumptions of these tests not only yield a zero in
$B^n$ but also a ``robust'' zero and a nontrivial $0$th well
group $U_0(f,r)$ for some $r>0$. 
Recently, topological degree  has been used for simplification of vector fields~\cite{Skraba}.

The first obstruction $\phi_f$ is the analog of the degree for underdetermined systems, that is, when $\dim K>n$ in our setting. To the best of our knowledge, this tool has not been algorithmically utilized.

\section{Computing lower bounds on well groups}\label{s:cap}
\heading{Homology theory behind the well groups.} For computing the approximation $\phi_f\frown H_k(X,A\cup B)$ of well
group $U_{k-n}(f)$ we only have to work with simplicial complexes and simplicial maps for which all homology theories satisfying the Eilenberg--Steenrod axioms are naturally equivalent. Hence, regardless of the homology theory $H_*$ used, we can do the computations in simplicial homology. Therefore the standard algorithms of computational topology \cite{EdelsbrunnerHarer:ComputationalTopology-2010}
and the formula for the cap product of a simplicial cycle and cocycle~\cite[Section 2.2]{prasolov} will do the job.
 
The need for a carefully chosen homology theory stems from the courageous claim that the zero set $Z$ of \emph{arbitrary} continuous perturbation \emph{supports} $\phi_f\frown \beta$ for any $\beta\in H_*(X,A\cup B)$, i.e. some element of $H_*(Z,B)$ is mapped by the inclusion-induced map to $\phi_f\frown \beta$. Without more restrictions on the perturbations, the zero sets can be ``wild'' non-triangulable topological spaces that can fool singular homology and  render this claim false and---to the best of our knowledge---make well groups trivial. See an example after the proof of Theorem~\ref{t:cap}.

For the purpose of the work with the general zero sets, we will require that our homology theory satisfies the Eilenberg--Steenrod axioms
with a possible exception of the exactness axiom, and these additional properties:
\begin{enumerate}
\item \label{e:surj} \emph{Weak continuity property:} for an inverse sequence of compact pairs $(X_0,B_0)\supset (X_1,B_1)\supset\ldots$ the homomorphism $H_*\varprojlim (X_i,B_i)\to \varprojlim H_*(X_i,B_i)$ induced by the family of inclusion $\varprojlim(X_i,B_i)=\bigcap(X_i,B_i)\hookrightarrow (X_j,B_j)$ \emph{is surjective.}
\item \emph{Strong excision:} Let $f\:(X,X')\to(Y,Y')$ be a map of compact pairs that maps $X\setminus X'$ homeomorphically onto $Y\setminus Y'$. Then $f_*\:H_*(X,X')\to H_*(Y,Y')$ \emph{is an isomorphism}.
\end{enumerate}

{\v{C}ech} homology theory satisfies these properties as well as the Eilenberg--Steenrod axioms with the exception of the exactness axiom, and coincides with simplicial homology for triangulable spaces~\cite[Chapter 6]{Wallace:2007}.

In addition, we need a cohomology theory $H^*$ that satisfies the Eilenberg--Steenrod axioms and is paired with $H_*$ 
via a~{cap product} $H^n(X,A)\otimes H_k(X,A\cup B)\stackrel{\frown}{\longrightarrow} H_{k-n}(X,B)$
that is natural\footnote{Naturality of the cap product means that if $f: (X,A\cup B,A)\to (X;A'\cup B',A')$ is continuous, then
$f_*(f^*(\tilde\alpha)\frown \beta)=\tilde\alpha\frown f_*(\beta)$ for any $\beta\in H_*(X,A\cup B)$ and $\tilde\alpha\in H^*(X',A')$.} 
and coincides with the simplicial cap product when applied to simplicial complexes. We have not found any reference
for the definition of cap product in \v Cech (co)homology, so we present our own construction in Appendix~\ref{a:Cech}.
\remove{However, if $(X,A)$ is a triangulable pair, then we may as well use simplicial cap product and identify $\phi_f\frown H_*(X,A\cup B)$
with the corresponding subgroup of our homology theory. After slight changes in the proof of Theorem~\ref{t:cap}, all cap products
could be only applied to triangulable spaces. Thus  Theorem~\ref{t:cap} would still hold under the assumption that the pair $(X,A)$ can be triangulated, that is, the expression $\phi_f\frown H_k(X,A\cup B)$ makes sense there. At least for computability results, this is no severe restriction. 
With this in mind, we might as well use the Steenrod homology theory of compact metrizable spaces~\cite{Milnor:1960}
with cap product defined simplicially on triangulable spaces. The advantage of Steenrod homology is that it satisfies the exactness axiom. We also believe that it is possible to pair it with a suitable cohomology theory by a cap product but we do not know how.}
\begin{proof}[Proof of Theorem~\ref{t:cap}] 
We need to show that for any map $g$ with $\|g-f\|\le r$, the image of the inclusion-induced map $$H_*(g^{-1}(0), B)\to H_*(X,B)$$
contains the cap product of the first obstruction $\phi_f:=f^*(\xi)$ with all relative homology classes of $(X, A\cup B)$. Let us first restrict to the less technical case of $g$ being a strict  $r$-perturbation, that is, $\|g-f\|<r$.

Let us denote $X_0:=X=|f|^{-1}[0,r]$ and $A_0:=A=|f|^{-1}(r)$. Next we choose a decreasing positive sequence $\epsilon_1>\epsilon_2>\ldots$
with $\lim_{i\to\infty} \epsilon_i=0$ and with  $\epsilon_1<r-\|f-g\|.$ Thus $X_1:=|g|^{-1}[0,\epsilon_1]\subseteq X_0$ and $A_0':=|g|^{-1}[\epsilon_2,\infty]\cap X_0\supseteq |g|^{-1}[\epsilon_2,\epsilon_1]$. Then we for each $i>0$ we define 
\begin{itemize}
\item $X_i:= |g |^{-1}[0, \epsilon_i]$, 
\item and its subspaces $A_i:=|g|^{-1}[\epsilon_{i+1},\epsilon_i]$, $A'_i:=|g|^{-1}[\epsilon_{i+2},\epsilon_i]$ and $B_i:=B\cap X_i$.
\end{itemize}
Note that $\bigcap_i X_i=g^{-1}(0)$, and consequently, $\bigcap_i B_i=g^{-1}(0)\cap B$. 
For any given $\beta\in H_k(X,A\cup B)$, our strategy is to find homology classes $\alpha_i\in H_{k-n}(X_i,B_i),$ with $\alpha_0=\phi_f\frown \beta$, that fit into the sequence of maps $H_{k-n}(X_0,B_0) \leftarrow H_{k-n}(X_1,B_1)\leftarrow\ldots$ induced by inclusions. This gives an element in $\varprojlim H_{k-n}(X_i,B_i)$, and consequently by the weak continuity property (requirement \ref{e:surj} above), we get the desired element $\alpha\in H_{k-n}(g^{-1}(0),B)$.

The elements $\alpha_i$ will be constructed as cap products. To that end, we need to obtain ``analogs'' of $\beta$ and for that we will need a more complicated sequence of maps. It is the zig-zag sequence
\begin{equation}\label{e:zigzag}X_0\stackrel{\id}{\to}X_0\stackrel{\incl}{\hookleftarrow}X_1 \stackrel{\id}{\to} X_1 \stackrel{\incl}{\hookleftarrow} X_2 \stackrel{\id}{\to}\cdots \end{equation}
that restricts to the zig-zags
\begin{equation}\label{e:zigzag-a}A_0\stackrel{\incl}{\hookrightarrow}A'_0\stackrel{\incl} { \hookleftarrow}A_1 \stackrel{\incl}{\hookrightarrow} A'_1\stackrel {\incl} {\hookleftarrow} A_2
\stackrel{\incl}{\hookrightarrow}\cdots \end{equation}
and
\begin{equation}\label{e:zigzag-b}A_0\cup B_0\stackrel{\incl}{\hookrightarrow}A'_0\cup B_0\stackrel {\incl} {\hookleftarrow}A_1\cup B_1 \stackrel{\incl}{\hookrightarrow} A'_1\cup B_1\stackrel {\incl} {\hookleftarrow} A_2\cup B_2
\stackrel{\incl}{\hookrightarrow}\cdots \end{equation}

The pair $(X_{i+1},A_{i+1}\cup B_{i+1})$ is obtained from $(X_i,A'_i\cup B_{i})$ by excision of $|g|^{-1}(\epsilon_{i+1},\epsilon_i]$, that is, $X_{i+1}=X_i\setminus |g|^{-1}(\epsilon_{i+1},\epsilon_i]$ and $A_{i+1}\cup B_{i+1}=(A'_i\cup B_i)\setminus |g|^{-1}(\epsilon_{i+1},\epsilon_i]$.   Hence by excision,\footnote{Because of our careful choice of the spaces $A_i$ and $A'_i$ we do not need the strong excision here. However, we do not know how to avoid it in the case when $\|g-f\|=r$.} each inclusion of the pairs $(X_i,A'_i\cup B_i)\hookrightarrow (X_{i+1},A_{i+1}\cup B_{i+1})$ induces isomorphism on relative homology groups. Therefore the zig-zag sequences \eqref{e:zigzag} and \eqref{e:zigzag-b} induce a  sequence

$$\begin{array}{c@{\ }c@{\ }c@{\ }c@{\ }c@{\ }c@{\ }c@{\ }c@{\ }c}
H_k(X_0,A_0\cup B_0)&\rightarrow&H_k(X_0,A'_0\cup B_0)&\cong&H_k(X_1,A_1\cup B_1)& \rightarrow&H_k(X_1, A'_1\cup B_1)&\cong&\cdots\\
\rotatebox{90}{$\in$}&& \rotatebox{90}{$\in$}&& \rotatebox{90}{$\in$}&& \rotatebox{90}{$\in$}&&\\[-1mm]
\beta_0:=\beta\!\!\!\!\!\!\!\!\!\!\!\!\!\!&&\beta'_0&&\beta_1&&\beta'_1&&\cdots
\end{array} $$
that can be made pointed by choosing the distinguished homology classes $\beta_i\in H_k(X_i,A_i\cup B_i)$ and $\beta'_i\in H_k(X_i,A'_i\cup B_i)$ that are the images of $\beta_0:=\beta\in H_k(X_,A\cup B)$ in this sequence.

Similarly, we want to construct a pointed zig-zag sequence in cohomology induced by \eqref{e:zigzag} and \eqref{e:zigzag-a}. The distinguished elements $\phi_i\in H^n(X_i,A_i)$ and $\phi'_i\in H^n(X_i,A'_i)$ are defined as the pullbacks of the fundamental cohomology class $\xi\in H^n(\R^n,\R^n\setminus\{0\})$ by the restrictions of $g$. Because of the functoriality of cohomology, $\phi_i$ and $\phi_i'$ fit into the sequence induced by \eqref{e:zigzag} and \eqref{e:zigzag-a}:
$$\begin{array}{c@{\ }c@{\ }c@{\ }c@{\ }c@{\ }c@{\ }c@{\ }c@{\
}c}
H^n(X_0,A_0)&\leftarrow&H^n(X_0,A'_0)&\rightarrow&H^n(X_1,A_1)& \leftarrow&H^n(X_1, A'_1)&\rightarrow&\cdots\\
\rotatebox{90}{$\in$}&& \rotatebox{90}{$\in$}&& \rotatebox{90}{$\in$}&&
\rotatebox{90}{$\in$}&&\\[-1mm]
{\phi_0}&&\phi_0'&&\phi_1&&\phi_1'&&\cdots\\
\end{array} $$
Since $g$ is an $r$-perturbation of $f$ and thus $g|_{(X,A)}$ is homotopic to $f|_{(X,A)}$ via the straight line homotopy, we have that $\phi_0=\phi_f\in H^n(X,A)$.

From the naturality of the cap product we get that the elements $\phi_i\frown \beta_i$ and $\phi_i'\frown \beta'_i$ fit into the sequence
$$\begin{array}{c@{\ }c@{\ }c@{\ }c@{\ }c@{\ }c@{\ }c@{\ }c@{\
}c}
H_{k-n}(X_0,B_0)&\stackrel{\id}{\cong}&H_{k-n}(X_0,B_0)& \leftarrow & H_{k-n}(X_1, B_1)& \stackrel{\id}{\cong}&H_{k-n}(X_1, B_1)& \leftarrow&\cdots\\
\rotatebox{90}{$\in$}&& \rotatebox{90}{$\in$}&& \rotatebox{90} {$\in$}&& \rotatebox{90}{$\in$}&&\\[-1mm]
\hskip-1.85mm{\phi_0}\frown\beta_0&&\hskip-1.85mm\phi_o'\frown \beta'_0&&\hskip-1.85mm\phi_1\frown\beta_1&& \hskip-1.85mm \phi_1' \frown \beta'_1 &&\cdots\\
\hskip0.3mm\rotatebox{90}{$=$}\\[-1mm]
\hskip-1.85mm{\phi_f}\frown\beta\\
\end{array} $$
that is induced by \eqref{e:zigzag}, that is, each $H_{k-n} (X_i,B_i) \stackrel{\id}{\cong} H_{k-n}(X_i,B_i)$ is induced by the identity $X_i\stackrel{\cong}{\to} X_i$ and each map $ H_{k-n}(X_i,B_i) \leftarrow H_{k-n}(X_{i+1}, B_{i+1})$ is induced by the inclusion $X_i\hookleftarrow X_{i+1}$. 
Hence $\alpha_i:=\phi_i\frown \beta_i$ are the desired elements and thus there is an element $\tilde\alpha:=(\alpha_0,\alpha_1,\ldots)$ in $\varprojlim H_{k-n}(X_i,B_i)$. 

We recall that  the weak continuity property of the homology theory $H_*$ assures the surjectivity of the the map \begin{equation}\label{e:surjectivity}(\iota_i)_{i\ge 0}\:H_{k-n}\Big(\bigcap X_i,B\Big)\to \varprojlim
H_{k-n}(X_i,B)\end{equation} where each component $\iota_i$ is induced by the inclusion $\bigcap_i X_i\hookrightarrow X_i$.  Let $\alpha \in H_{k-n}(g^{-1}(0), B)$ be arbitrary preimage of $\tilde\alpha$ under the surjection \eqref{e:surjectivity}. By construction, $\alpha$ is mapped to $\alpha_0=\phi_f\frown \beta$ by the map $\iota_0$.

It remains to prove the theorem in the case when $\|g-f\|=r$. The proof goes along the same lines with only the following differences:
\begin{itemize}
\item  For arbitrary decreasing sequence $1=\epsilon_0>\epsilon_1>\epsilon_2> \ldots$ with $\lim \epsilon_i=0$ we define $h_i:=\epsilon_i f+ (1-\epsilon _i)g$ for $i\ge0$.
We will furthermore
need that $2\epsilon_{i+1}>\epsilon_i$ for every $i\ge0$. Let 
$$\begin{array}{c@{\,}l}
X_i&:=|h_i|^{-1}[0,\epsilon_ir], \\[-.5mm]
\ \rotatebox{90}{$\subseteq$} \\[-1.5mm]
A'_{i}&:=\{x\in X\:|h_i(x)|\le \epsilon_i r\text{ and }|h_{i+1}(x)| \ge \epsilon_{i+1} r\}\text{ and }\\[-.5mm]
\ \rotatebox{90}{$\subseteq$} \\[-1.5mm]
A_i&:=|h_i|^{-1}(\epsilon_i r).\end{array}$$
We have $A_i\subseteq A'_i$ because by definition $\|h_i-h_{i+1}\|\le (\epsilon_i -\epsilon_{i+1})r$ and thus
\(|h_i(x)|= \epsilon_i r\) implies \(|h_{i+1}(x)|\geq \epsilon_{i+1}
r\). Similarly $A_{i+1}\subseteq A'_i$ and $X_{i+1}\subseteq X_i$.
Therefore as before, the zig-zag sequence \eqref{e:zigzag} restricts to \eqref{e:zigzag-a} and \eqref{e:zigzag-b}. 

\item The homology classes $\beta_i$ and $\beta'_i$  are defined as above. We only need to use the strong excision for the inclusion $(X_i,A'_i\cup B_i)\hookleftarrow (X_{i+1},A_{i+1}\cup B_{i+1})$.

\item We define the cohomology classes $\phi_i:=h_i^*(\xi)$ and $\phi_i':= h_{i+1}^*(\xi)$. We only need to check that $h_i$ is homotopic to $h_{i+1}$ as a map of pairs $(X_i,A'_i)\to (\R^n,\R^n\setminus\{0\})$. Indeed, they are homotopic via the straight-line homotopy since $|h_{i+1}(x)|\ge \epsilon_{i+1}r$ implies $|h_i(x)|\ge \epsilon_ {i+1}r- (\epsilon_i-\epsilon_{i+1})r=(2\epsilon_{i+1}-\epsilon_i)r >0$.
We used the inequality $2\epsilon_{i+1}>\epsilon_i$ which was our requirement on the sequence $(\epsilon_i)_{i>0}$.
We also have $\phi_0=\phi_f$ as $h_0=f$ and $(X_0,A_0)=(X,A)$.

\item We continue by defining cap products $\alpha_i$, their limit $\tilde \alpha$ and its preimage $\alpha$ under the surjection $H_{k-n}(\bigcap_i X_i,B)\to \varprojlim_i H_{k-n}(X_i,B)$. To finish the proof we claim that $\bigcap_i X_i=g^{-1}(0)$. Indeed, $g(x)=0$ implies $h_i(x)\le \|h_i-g\|=\epsilon_i r$ for each $i$ and $g(x)>0$ implies $h_i(x)>0$ for  $i$ such that $2\epsilon_i r<|g(x)|$.
\end{itemize}
\end{proof}

The surjectivity of (\ref{e:surjectivity}) and the strong excision is not only a crucial step for Theorem~\ref{t:cap} but implicitly also
for the results stated in~\cite[p. 16]{interlevel}. If we defined well groups by means of singular homology,
then even in a basic example $f(x,y)=x^2+y^2-2$ and $r=1$,
the first well group $U_1(f,r)$ would be trivial. The zero set of any $1$-perturbation $g$ is contained in 
the annulus $X:=\{(x,y):\,\,1 \leq x^2+y^2\leq 3\}$ and the two components of $\partial X$ are not in the same connected components
of $\{x\in X:\,g(x)\neq 0\}$. However, we could construct a ``wild'' $1$-perturbation $g$ of $f$ such that $g^{-1}(0)$ 
is a Warsaw circle~\cite{Mardevsic:1997} which is, roughly speaking, a circle with infinite length, trivial first singular homology, but nontrivial \v{C}ech homology.
Thus \v{C}ech homology serves as a better theoretical basis for
the well groups. Another solution to avoid problems with wild zero sets would be to restrict ourselves to ``nice'' perturbations, for example
piecewise linear or smooth and transverse to~$0$. 
Such approach would lead, to the best of our knowledge, to identical 
results.

\begin{proof}[Proof of Theorem~\ref{o:comp}] Under the assumption on computer representation of $K$ and $f$, the pair $(X,A)$ is homeomorphic to a computable simplicial pair $(X',A')$ such that
$X'$ is a subcomplex of a subdivision $K'$ of $K$~\cite[Lemma
3.4]{nondec}. 
Therefore, the induced triangulation $B'$ of $B\cap X'$ is a
subcomplex of $X'$. Furthermore, a simplicial approximation \(f' \:A'\to S'\) of $f|_A\: A\to S^{n-1}$ can be
computed. The computation is implicit in the proof of Theorem
1.2 in~\cite{nondec} where the sphere $S^{n-1}$ is approximated
by the boundary $S'$ of the $n$-dimensional cross polytope $B'$. The simplicial approximation
$(X',A')\to (B',S')$ of $f|_X$ can be constructed consequently by sending
each vertex of $X\setminus A$ to an arbitrary point in the interior
of the cross polytope, say $0\in\R^n$. 
The pullback of a cohomology class can be computed by standard
algorithms. Therefore $\phi_f$ and $H_*(X,B)$ can be computed and
the explicit formula for the cap product in~\cite[Section 2.1]{prasolov}
yields the computation of $\phi_f\frown H_*(X,B)$. 
All this can be done without any restriction on
the dimensions of the considered simplicial complexes.  
\end{proof}
\heading{Well diagram associated with $\phi\frown H_*(X,A\cup B)$.}
Let $r_1>r_2>0$ and let $X_1$, $X_2$, $A_1$, $A_2$ be $|f|^{-1}[0,r_1]$, $|f|^{-1}[0,r_2]$, $|f|^{-1}\{r_1\}$, $|f|^{-1}\{r_2\}$ respectively,
$\phi_1$, $\phi_2$ be the respective obstructions.
Further, let $A_1':=|f|^{-1}[r_2,r_1]$ and $\phi_1'=f^*(\xi)\in H^n(X_1, A_1')$ be the pullback of the fundamental class 
$\xi\in H^n(\R^n, \R^n\setminus\{0\})$. The inclusions $(X_1, A_1)\subseteq (X_1, A_1')\supseteq (X_2, A_2)$
induce cohomology maps that take $\phi_1'$ to $\phi_1$ resp. $\phi_2$.
Let us denote, for simplicity, by $V_1$ the group $\phi_1\frown H_*(X_1, A_1\cup B)$, $V_2:=\phi_2\frown H_*(X_2, A_2\cup B)$
and $V_1':=\phi_1'\frown H_*(X_1, A_1'\cup B)$. Further, let $U_1$ resp. $U_2$ be the well groups $U(f,r_1)$ resp. $U(f, r_2)$.
\begin{wrapfigure}{l}{0.6\textwidth}
\begin{equation*}
\hskip-0.6cm
\label{e:well diagram}
\begin{diagram}
V_2 & & \lTo^{\iota_{12}} & &  V_1  \\
\dInto & & & & \dInto \\
U_2 & \rTo^{a} & U_2/(U_2\cap \ker i_{21})& \lTo^{b} & U_1\\
\dInto &&& \luTo^{\simeq } & \dInto \\
H_*(X_{2}, B) & \rTo^{i_{21}}& H_*(X_{1}, B) & \lInto & i_{21}(U_2)\\
\end{diagram}.
\end{equation*}
\end{wrapfigure}
In this section, we analyze the relation between $V_1$ and $V_2$.
First let $i_1$ be a map from $V_1$ to $V_1'$ that maps $\phi_1\frown\beta_1$ to
$\phi_1'\frown i_*(\beta_1)$. By the naturality of cap product, $\phi_1\frown\beta_1=\phi_1'\frown i_*(\beta_1)$,
so $i_1$ is an inclusion.
By excision, there is an inclusion-induced isomorphisms 
$i_1': H_*(X_2, A_2\cup B)\stackrel{\sim}{\rightarrow} H_*(X_1, A_1'\cup B)$ 
and its inverse induces an isomorphism 
$i_2: V_1' \stackrel{\sim}{\to} V_2$
by mapping $\phi_1'\frown \beta_1'$ to $\phi_2\frown (i_1')^{-1}(\beta_1')$.
The composition $i_2\circ i_1=:\iota_{12}$ is a homomorphism from $V_1$ to $V_2$. 
Being the composition of an inclusion
and an isomorphism, $\iota_{12}$ is an injection and one easily verifies that the inclusion-induced map 
$i_{21}: H_*(X_2, B){\to} H_*(X_1, B)$ satisfies $i_{21}\circ\iota_{12}=\mathrm{id}|_{V_1}$.
It follows that $\{V(r_i), \iota_{i,i+1}\}_{r_i>r_{i+1}}$ is a persistence module consisting of shrinking abelian groups
and injections $V_{i}\to V_{i+1}$ for $r_i>r_{i+1}$. 
The relation between $\iota$ and well diagrams described in~\cite{Edelsbrunner:2011} is reflected by the commutative diagram above.

\remove{The rank of $U(r)$ resp. $V(r)$ can only decrease with increasing $r$. In~\cite{Edelsbrunner:2011}, authors encode the properties of well groups
to a well diagram that consists of pairs $\{(r_j, \mu_j)\}$ where $r_j$ is a number in which the rank of $U$ decreases by $\mu_j\in\N$. 
Using computable information about $\{V(r)\}$, we may define a~diagram consisting of pairs $(r_j', \mu_j')$ 
where the rank of $V(r)$ decreases in $r_j'$ by $\mu_j'$. 
This is a subdiagram of the well diagram in the following sense: each $r_k'$ is then contained in $\{r_j\}_j$ and $\mu_k'\leq \mu_k$.}

\heading{The idea behind the proof of Theorem~\ref{t:Poincare}.} In the special case when $X$ is a smooth $m$-manifold with $A=\partial X$, 
the zero set of any smooth $r$-perturbation $g$ transverse to $0$ 
is an $(m-n)$-dimensional smooth submanifold of $X$. It is not so difficult to show that its fundamental class $[g^{-1}(0)]$ is mapped by the inclusion-induced map to $\phi_f\frown [X]$,
where $[X]\in H_m(X,\partial X)$ is the fundamental class of $X$. If $g^{-1}(0)$ is connected, then
$H_{m-n}(g^{-1}(0))$ is generated by its fundamental class and we immediately obtain the reverse inclusion
$\phi_f\frown H_m (X,A)\supseteq U_{m-n}(f,r)$. The nontrivial part in the proof of Theorem~\ref{t:Poincare}
is to show that in the indicated dimension range, we can find a perturbation $g$ so that $g^{-1}(0)$ is connected.
The full proof is in  Appendix~\ref{s:Poincare}.

\section{Incompleteness of well groups}\label{s:incompleteness}
In this section, we study the case when the first obstruction $\phi_f$ is trivial and thus the map $f|_A$ can be extended to a map $f^{(n)}\:X^{(n)}\to S^{n-1}$ on the $n$-skeleton $X^{(n)}$ of $X$. Observation~\ref{o:triv} (proved in Appendix~\ref{a:triv}) implies that the only possibly nontrivial well groups are $U_j(f,r)$
for $j\leq m-n-1$. 

The following lemma summarizes the necessary tools for the constructions of this section. They directly follow from Lemma~\ref{l:perturb-ext} in Appendix~\ref{a:characterization} and from \cite[Lemma 3.3]{nondec}.
\begin{lemma}\label{l:ext-pert}
Let $f\:K\to \R^n$ be a map on a compact Hausdorff space, $r>0$, and let  us denote the
pair of spaces $|f|^{-1}[0,r]$ and $|f|^{-1}\{r\}$ by $X$ and
$A$, respectively. Then \begin{enumerate}
\item \label{n:one}
for each extension $e\:X\to\R^n$ of $f|_A$ we can find a strict $r$-perturbation $g$ of $f$ with $g^{-1}(0)= e^{-1}(0)$;
\item \label{n:two}
for each $r$-perturbation $g$ of $f$ without a root there is an extension $e\:X\to \R^n\setminus\{0\}$ of $f|_A$ (without a root).
\end{enumerate}
\end{lemma}

In the following we want to show that well groups can fail to distinguish between maps with intrinsically different families of zero sets. Namely, in the following examples we present maps $f$ and $f'$ with  $U_0(f,r)=U_0(f',r)=\Z$ for each $r\leq1$ and $U_i(f,r)=U_i(f,r)=0$ for each $r\leq1$ and $i>0$. However, $Z_r(f)$ will be significantly different from $Z_r(f')$.


\begin{proof}[{Proof of Theorem~\ref{t:torus}}]
We have that $B=\emptyset$ and $K=S^j\times B^i$, where $B^i$ is represented by the unit ball in $\R^i$ and $j=m-i$. Let the maps $f,f': K\to\R^n$ be defined by 
$$
f(x,y):=|y| \,\varphi(x,y/|y|)
\quad \text{ and }\quad 
f'(x,y):=|y|\varphi' (x,y/|y|)
$$
where $\varphi,\varphi'\:S^j\times S^{i-1}\to S^{n-1}\subseteq\R^n$ are defined by
\begin{itemize}
\item $\varphi(x,y):=\mu(y)$ where $\mu: S^{i-1}\to S^{n-1}$ is an arbitrary nontrivial map. 
\item $\varphi'$ is defined as the composition $S^j\times S^{i-1}\to S^{m-1}\stackrel{\nu}{\to} S^{n-1}$ where the first map is the quotient map $S^j\times S^{i-1}\to S^j\wedge S^{i-1}\cong S^{m-1}$ and $\nu$ is an arbitrary nontrivial map.
In other words, we require that the composition $\varphi'\Phi$---where $\Phi$ denotes the characteristic map of the $(m-1)$-cell of $S^j\times S^{i-1}$---is equal to the composition $\nu q$, where $q$ is the quotient map $B^{m-1}\to B^{m-1}/(\partial B^{m-1})\cong S^{m-1}$.\end{itemize}

\heading{Well groups computation.} Next we prove that the well groups of $U_*(f,r)$ and $U_*(f',r)$
are the same for $r\in(0,1]$, namely, nonzero only in dimension
$0$, where they are isomorphic to $\Z$. 
We obviously have $X=S^j\times\{y\in\R^i\:|y|\le r\}\simeq S^j\times B^i$ and $A=\partial X$ for both maps. The restriction $f|_A$ and $f'|_A$ are equal to $\varphi$ and $\varphi'$ (after normalization). We first prove that $U_0(f,1)\cong U_0(f',1)\cong\Z$. This fact follows from $H_0(X)\cong\Z$, from non-extendability of $\varphi$ and $\varphi'$ and from Lemma~\ref{l:ext-pert} part~\ref{n:two} (or \cite[Lemma 3.3]{nondec}). \begin{lemma}\label{l:nonext}
The map $\varphi'$ cannot be extended to a map $X\to S^{n-1}$.
\end{lemma}
We postpone the proof to Appendix~\ref{a:nonext}.
Since the map $\mu\:S^{i-1}\to S^{n-1}$ cannot be extended to $B^i\supset S^{i-1}$, also $\varphi$ cannot be extended to $X$. 

Since then only the $j$th homology group of $X$ is nontrivial, the remaining task is to show that $U_j(f,1)\cong U_j(f',1) \cong 0$. We do so by presenting two $r$-perturbations $g$ and $g'$ of $f$ and $f'$, respectively:\begin{itemize}
\item  $g(x,y):=f(x,y)-rx=|y|\mu(y/|y|)-rx$ where we consider $S^j\subseteq\R^{j+1}$ as a subset of $\R^n$ naturally embedded in the first $j+1$ coordinates (here we need that $j=m-i<n$).\item  We first construct an extension  $e'\:X\to \R^n$ of $\varphi' =f'|_A$ and then the $r$-perturbation $g'$ is obtained by Lemma~\ref{l:ext-pert} part~\ref{n:one}. The extension $e'$ is  defined as constant on the single $i$-cell of $X$, that is, $e'(x_0,y)$ is put equal to the basepoint of $S^{n-1}\subseteq \R^n$. On the remaining  $m$-cell $B^m\cong \{z\in\R^m\:|z|\le 1\}$ of $X$ we define $e'(z):=|z|e'(z/|z|)$, where each point $z$ is identified with a point of $X$ via the characteristic map $\Psi_1\:B^m\to X$ of the   $m$-cell $B^m$.\footnote{Thus the formal definition is $e'(\Psi_1(z)):=|z|e'\big(\Psi_1(z/|z|)\big)$.}
 \end{itemize}
By definition the only root of $g'$ is the single point $\Psi_1(0)$ of the interior of $X$. Therefore $U_j(f,1)\cong 0$. Note that the role of $\Psi_1(0)$ could be played by an arbitrary point in the interior of $X$.\footnote{ With more effort we could show that for \emph{any} point $z$ of $X$ there is an $r$-perturbation of $f'$ with $z$ being its only zero point.}  

The zero set $g^{-1}(0)=\{(x,y):\,|y|=r$ and  $\mu(y/|y|)=x\}$ is by definition homeomorphic
to the pullback (i.e., a limit) of the diagram
\begin{equation}\label{e:pullback}
\xymatrix{& S^{i-1}\ar[d]^{\mu}\\
S^j\ar[r]^{\iota}&S^{n-1}}\end{equation}
where $\iota$ is the equatorial embedding, i.e., sends each element $x$ to $(x,0,0,\ldots)$.
In plain words, the zero set is the $\mu$-preimage of the equatorial $j$-subsphere of $S^{n-1}$. We will prove that under our assumptions on dimensions, this  is the
 $(m-n)$-sphere $S^{m-n}$. Then from $m-n>m-i=j$ it will follow that $H_j(g^{-1}(0))\cong
0$ which proves Theorem~\ref{t:torus}. 

The topology of the pullback is particularly easy to see in the case when $j=n-1$ and $\iota$ is the identity. 
There it is simply the domain of $\mu$, that is, $S^{i-1}$ where $i-1=m-j-1=m-n$.

In the general case, the only additional tool we use to identify the pullback is the Freudenthal suspension theorem.
The pullback is homeomorphic to the $\mu$-preimage
of the
equatorial subsphere $S^{m-i}\subseteq S^{n-1}$.
By Freudenthal suspension theorem $\mu$ is homotopic
to an iterated suspension $\Sigma^a\eta$ for some $\eta\:S^{i-1-a}\to
S^{n-1-a}$ assuming 
$i-1-a\le2(n-1-a)-1$.
We want to choose $a$ so that  $n-1-a=m-i$ and thus images 
$\Im(\eta)=S^{n-1-a}$ and $\Im(\iota)=S^j\subseteq S^{n-1}$ coincide (since $j=m-i$ by definition).
The last inequality 
with the choice  $a=n-1-m+i$  is equivalent to the bound $i\leq (m+n-1)/2$ from the hypotheses of the theorem.
In our example, we may have chosen $f$ in such a way that  $\mu=\Sigma^a\eta$. But even for the choices of $\mu$ only homotopic to $\Sigma^a\eta$
we could have changed $f$ on a neighborhood of $\partial
K$ by a suitable homotopy. 
To finish the proof we use the fact that, by the definition of
 suspension, the $\mu$-preimage of $S^{m-i}\subseteq S^{n-1}$
is identical to the $\eta$-preimage of $S^{m-i}$, that is $S^{i-1-j}=S^{m-n}$.

\heading{Difference between $Z_r(f)$ and $Z_r(f')$.}
Because the map $\mu$ is homotopically nontrivial, the zero set of each extension $e\:X\to\R^n $ of $f|_A$ intersects each ``section''  $\{x\}\times B^i$ of $X$. By  Lemma~\ref{l:ext-pert} part~\ref{n:two} (or \cite[Lemma 3.3]{nondec}) applied to each
restriction $f|_{\{x\}\times B^i}$,  the same holds for $r$-perturbations
$g$ of $f$ as well. In other words, the formula ``for each $x\in S^j$ there is $y\in B^i$ such that $f(x,y)=0$'' is \emph{satisfied robustly}, that is
$$\forall Z\in Z_r(f):\forall x\in S^j: \exists y\in B^i: (x,y) \in Z$$ is satisfied. 
The above formula is obviously not true for $f'$ as can be seen on the $r$-perturbations $g'$. In particular, for every $r\in(0,1]$ the family $Z_r(f')$ contains a singleton.\end{proof}

\heading{Robust optimization.}
As an example of another relevant property of $Z_r(f)$ not captured by the well groups, we mention the following. For any given $u\:K\to\R$, we may want to know what is the \emph{$r$-robust maximum 
of $u$ over the zero set of $f$}, i.e., $\inf_{Z\in Z_r(f)}\max_{z\in Z} u(z)$. Let, for instance, $u(x,y)=u(x)$ depend on the first coordinate only. Then the $r$-robust maximum for $f$ is equal to $\max_{x\in S^j} u(x)$ as follows from the discussion in the previous paragraph. On the other hand, the $r$-robust maximum for $f'$ is equal to $\min_x u(x)$ and is attained in $g'$ when we set the value $\Psi_1(0):=(\argmin_{x\in S^j}u(x),0)$ from the proof above. This holds for $r$ arbitrarily small. The robust optima constitutes another and, in our opinion, practically relevant  quantity whose approximation cannot be derived from well groups. 

\remove{
\heading{Further remarks on Theorem~\ref{t:torus}.} 
We first want to indicate that in some sense the maps $f$ and $f'$ are no peculiar examples but rather typical choices. 
More precisely, we assume that $r>0$ is fixed and that $X=S^j\times B^i$ and $A=\partial X$. (Note that in the natural cell structure of $X$  there is only one $i$-dimensional and one $(i+j)$-dimensional cell outside of $A$.) It can be easily proved that under these assumptions the maps $f$ and $f'$ can be chosen arbitrarily in such a way that \begin{itemize}
\item $f|_A$ cannot be extended to $X^{(i)}$ (it extends to $
X^{(i-1)}$ trivially as $A=X^{(i-1)}$) and
\item $f'|_A$ extends to $X^{(i)}$ but not to $X$.\footnote{
The only remaining category consists of those $f''$ where $f''|_A$
extends to whole of $X$, i.e., $U_*(f,r)\cong 0$, or equivalently,
$\emptyset\in Z_r(f)$.}
\end{itemize}
The only addition needed to prove this more general version is in the computation of $U_{m-i}(f,r)$. For that we can either use Theorem~\ref{t:triviality} when $i< (m+n)/2$ or enhance the proof of Theorem~\ref{t:torus} when $i=(m+n)/2$ which we omit here. 
Note that the nonextendability properties of $f$ and $f'$ require nontriviality of the homotopy groups of spheres as in the hypothesis of Theorem~\ref{t:torus}. Then only for the requirement $i>n$ we know that is strict. The other two inequalities are used to find the map $\iota$ such that the pullback~\eqref{e:pullback} is connected enough. The inequality $i<(m+n+1)/2$ can be relaxed to requiring the existence of $[\mu]\in\pi_{i-1}(S^{n-1})$ such that $[\mu]=\Sigma^a \eta$ for $a$ sufficiently large as stated in the proof.

Finally, we remark that the same incompleteness results could be achieved for even more realistic domain $K=B^j\times B^i\cong B^m$. We only need to choose $f$ and $f'$ with $X=B^j\times B^i_r$ and $A=B^j\times (\partial B^i_r)$ and with the same properties as above. Then for the natural choice $B=\partial K$ and under the same hypotheses, both well groups will be equal.

\heading{Sketch of the proof of Lemma~\ref{l:nonext}.} 
The ultimate claim is that $\varphi'$ cannot be extended to the $m$-cell of $X$ no matter how the extension on the $i$-cell was chosen. To this end, we need two properties of the obstruction to extendability on the $m$-cell (which is an element of $\pi_{m-1} (S^{n-1})$): \begin{enumerate}

\item First, that the obstruction is independent of the choice of the extension on the $i$-cell. This essentially follows from the bilinearity of the Whitehead product $\pi_{i}(S^{n-1})\otimes \pi_{m-i}(S^{n-1})\to \pi_{m-1} (S^{n-1})$, namely, that the Whitehead product of a trivial element with an arbitrary element is again a trivial element.  
\item Second, that the obstruction depends linearly on the choice of the element $[\nu]\in \pi_{m-1}(S^{n-1})$ in the definition of the map $\varphi'$. This amounts to the basic obstruction theory and the cell structure of the solid torus.
\end{enumerate}
The full proof is presented in Appendix~\ref{a:nonext}.
}

\heading{Acknowledgements.} We are grateful to Ryan Budnay, Martin
\v Cadek,
Marek Filakovsk\'y, Tom Goodwillie, Amit Patel, Martin Tancer,
Luk\'a\v s Vok\v{r}\'inek and Uli Wagner for useful discussions.

\bibliographystyle{plain}
\bibliography{sratscha,Postnikov}

\newpage
\appendix

\section{The nonextendability proof (proof of Lemma~\ref{l:nonext})}
\label{a:nonext}
\begin{proof}[Proof of Lemma~\ref{l:nonext}] 

As our ultimate claim will be that $\varphi'$ cannot be extended to the $m$-cell of $X=S^j\times B^i$, we will need to analyse the gluing map of that cell. In particular, we need to establish its relation to the gluing map of the $m$-cell of $T=S^j\times S^i$. In the first row of the following commutative diagram we have the factorization of the characteristic map $\Psi$ of the $m$-cell of $T$ into the characteristic map $\Psi_1$ of the $m$-cell of $X$ and another quotient map $\Psi_2$.

$$\xymatrix{
B^j\times B^i\ar[r]^-{\Psi_1}&X=(B^j/\partial B^j)\times B^i \ar[r]^-{\Psi_2}&T=(B^j/\partial B^j)\times (B^i/\partial B^i)\\
\partial(B^j\times B^i)\ar@{^(->}[u]\ar[r]^-{\psi_1} &
X^{(m-1)} \ar@{^(->}[u]\ar[r]^-{\psi_2} &
T^{(m-1)}=S^j\vee S^i\ar@{^(->}[u]
}
$$
Note that above we identify spheres with the quotients of balls by their boundary. The restriction of each characteristic map
to the boundary $\partial (B^j\times B^i)$ gives the respective gluing maps as is shown in the second row. We have that  $X^{(m-1)} = \big(S^j\times \partial B^i\big) \cup \big(\{*\}\times B^i\big)$ and\footnote{Having arbitrary space in mind, the sign $*$ will denote its basepoint. In our case, it will always be the single $0$-cell of the space.} indeed the quotient map $\Psi_2$ (which identifies each $\{x\}\times \partial B^i$ into a point $(x,*)$) maps $X^{(m-1)}$ to $S^j\vee S^i$. The  crucial consequence is that the gluing map $\psi$ of the $m$-cell in $T$ is the composition of the gluing map $\psi_1$ of the $m$-cell in $X$ and the restriction $\psi_2$ of the quotient map $\Psi_2$ as above.
 
Another tool that we need is the \emph{Whitehead product} \cite[Chapter X, 7.2]{Hatcher} for which we need the explicit construction. Again, $\psi\:S^{m-1}\to S^j\vee S^{i}$ denotes the gluing map of the $m$-cell in $T=S^j\times S^{i}$. 
Then the homotopy class of the composition $S^{m-1}\stackrel{\psi}{\to}S^j\vee S^{i} \stackrel{\omega_1\vee\omega_2}{\longrightarrow}S^{n-1}$ defines the Whitehead product of arbitrary elements $[\omega_1]\in\pi_{j}(S^{n-1})$ and $[\omega_2]\in\pi_{i}(S^{n-1})$. We will use the bilinearity of the Whitehead product, especially, that the product of the trivial element $[\const]$ and arbitrary $[\omega]$ is trivial.   

Let us assume that $$h\:\underbrace{e^i\cup A}_{X^{(i)}\cup A}\to S^{n-1}$$ is an extension of $\varphi'$ on the unique $i$-cell $e^i$ of $X$ . The map $h$ can be extended to $X$ if and only if there is a nullhomotopy for the composition $h\psi_1\:S^{m-1}\to S^{n-1}$ where again $\psi_1\:S^{m-1}\to X^{(m-1)}=X^{(i)}\cup A$ is the gluing map of the $m$-cell of $X$. Roughly speaking, the key difficulty is that $\varphi'$ can be extended on $e^i$ in various essentially different ways (whenever $\pi_i(S^{n-1})$ is nontrivial). The key observation is that this choice does not influence the homotopy class of $h\psi$ and that it is always equal to the above chosen nontrivial element $[\nu]\in\pi_{m-1}(S^{n-1})$ up to a sign.

Towards that end, let us form an auxiliary map $h'\:e^{i}\cup A\to S^{n-1}$ that is constant on $\partial X$ and equal to $h$ on the unique $i$-cell of  $X$. We want to show first, that $h'\psi_1$ is homotopically trivial, and second, that $[h\psi_1]=[h'\psi_1]\pm\nu$. 
\begin{enumerate}
\item Since $h'$ is constant on $\{x\}\times (\partial B^i)$ for each $x\in S^j$, it factors through the corresponding quotient $S^j\vee S^i$ of $X^{(i)}$ as follows (in the following the factorization is preceded by $\psi_1$):
$$\xymatrixcolsep{1.5cm}\xymatrix{S^{m-1}\ar[r]^-{\psi_1}&X^{(m-1)}=X^{(i)}\ar[r]^-{\psi_2}&
S^j \vee S^i\ar[r]^-{\const\!\vee \omega}&S^{n-1}}$$
Here $\const$ denotes the constant map and $\omega$ is the map determined by $h$ (or equivalently by $h'$) on $\{*\}\times B^i$. Since by the above considerations $\psi_2\psi_1=\psi$, the composition $h'\psi_1$ is equal to  $(\const\!\vee \omega)\psi$---the representative of the Whitehead product of $[\const]$ and $[\omega]$. By the bilinearity of the Whitehead product, $h'\psi_1$ is  trivial.

\item The second claim---$[h\psi_1]=[h'\psi_1]\pm\nu$---follows from basic obstruction theory. This claim follows from the  fact, that for any pair of maps $h$ and $h'$ that agree on $X^{(m-2)}$, the difference of their obstruction cocycles $z^m_h-z^m_{h'}\in Z^m\big(X,\pi_{m-1}(S^{n-1})\big)$ equals the coboundary of the difference cochain $d_{h,h'}\in C^{m-1}\big(X,\pi_{m-1}(S^{n-1} )\big)$. To get the conclusion, we need two ingredients: that the coboundary map is an isomorphism and that the difference cochain is nontrivial, namely, that it assigns $\nu$ to the $ (m-1)$-cell of $X$. The first ingredient was already shown in the first paragraph of this proof. 
Since the cellular chain structure of $X$ is rather simple---having one generator in both dimensions $m$ and $m-1$---we rephrase everything in an elementary language below.

The first ingredient is that the degree $d$ of the composition $$S^{m-1}\stackrel{\psi_1}{\to} X^{(m-1)}\stackrel{}{\to}X^{(m-1)}/(S^j\vee B^i)\cong S^{m-1}$$ is equal to $\pm 1$. The second ingredient is that, once we denote  the characteristic map
of the $(m-1)$-cell of $X$ by $\Phi$, the \emph{difference map} of $h'\Phi$ and $h\Phi$ equals $\pm\nu$. The difference map of any given maps $f\:B^{m-1}\to S^{n-1}$ and $g\:B^{m-1}\to S^{n-1}$ with $f|_{\partial B^{m-1}}=g|_{\partial B^{m-1}}$ is defined as  $\delta_{f,g}:=f\cup_{\partial B^{m-1}} g\:S^{m-1}\to S^{n-1}$. In words, $f$ defines $\delta_{f,g}$ on the northern hemisphere and $g$ defines $\delta_{f,g}$ on the southern hemisphere. Because there are factorizations $h\Phi=\varphi'\Phi=\nu q$  and\footnote{We have that $h\Phi=h|_A\Phi$ and $h|_A=\varphi'$ and we remind that the equality $\varphi'\Phi=\nu q$, where $q$ is the quotient map $B^{m-1}\to B^{m-1}/(\partial B^{m-1})$, was required in the definition of $\varphi'$} $h'\Phi=\const=\nu\const$ through $S^{m-1}$, we have that $\delta_{h\Phi, h'\Phi}= \nu\delta_{q,\const}$. Obviously, the map $\delta_{q,\const}$ has degree $\pm 1$ and thus the second ingredient holds.

By the definition of the addition in $\pi_{m-1}(S^{n-1})$, we have that $[h\psi_1]=[h'\psi]\pm d[\delta_{h\Phi,h'\Phi}]$. 
\end{enumerate}
This concludes the proof of Lemma~\ref{l:nonext} and thus the proof of Theorem~\ref{t:torus} as well.
\end{proof}

\section{Proof of Theorem~\ref{t:Poincare}}\label{s:Poincare}
\heading{Overview.} The proof of Theorem~\ref{t:Poincare} will be divided into several steps. 
Theorem~\ref{t:cap} implies one inclusion and for the other, it is sufficient to find a smooth $r$-perturbation
$g$ of $f$ such that $0$ is a regular value of $g$ and the homology image of $g^{-1}(0)$ in $X$ generates the cap product $\phi_f\frown H_m(X,\partial X)$.
First we show that the general case easily reduces to the case where $X$ is connected.
In the next step, we describe the $(m-n)$th homology of zero sets of perturbations transverse to zero: we prove that if $0$ is a regular value of a strict $r$-perturbation $g$ of $f$, then the Poincar\'e dual of $\phi_f$ equals the image of the fundamental class of the submanifold $g^{-1}(0)$ (Lemma~\ref{l:Poincare}). 
If $g^{-1}(0)$ is connected, then the fundamental class of $g^{-1}(0)$ generate its top homology $H_{m-n}(g^{-1}(0))$.
In this way, we reduce the proof of~Theorem~\ref{t:Poincare} 
to the statement that if $n+1\leq \dim X\leq 2n-3$ and $X$ is connected, then there exists some smooth strict $r$-perturbation $g$ of $f$ 
such that $0$ is a regular value of $g$ and $g^{-1}(0)$ is connected.
To prove this, we introduce the notion of framed submanifolds and show that if a framed $(m-n+1)$-submanifold $W$
has framed boundary consisting of $S\sqcup f^{-1}(0)$, then $S$ is the zero set of some smooth strict $r$-perturbation $g$ having $0$ as regular value (Lemma~\ref{l:framing}).
Finally, we show that there exists a framed submanifold $W$ and a \emph{connected} framed $(m-n)$-submanifold
$S$ of $X$  s.t. $\partial W=f^{-1}(0)\sqcup S$ (Lemma~\ref{l:connected_sum}). 
\heading{Reduction to the case of connected $X$}.
Assume that Theorem~\ref{t:cap} holds for $X$ connected. The compact space $X$ can only contain finitely many
connected components, say $X_1,\ldots, X_k$. Then $H_m(X, \partial X)\simeq \sum_i H_m(X_i, \partial X_i)$
and $\phi_f\frown H_m(X, \partial X)\simeq \sum_j \iota_j^*\phi_f\frown H_m(X_j, \partial X_j)$ where $\iota_j: X_j\hookrightarrow X$
is the inclusion. If we assume that Theorem~\ref{t:Poincare}
holds for $X$ connected, we may use it for $f|_{X_i}: X_i\to\R^n$ and get that
\begin{equation}
\label{e:components}
\bigcap_{g|_{X_i}:\,\|g|_{X_i}-f|_{X_i}\|\leq r} \mathrm{Im}\big(H_{m-n}(g|_{X_i}^{-1}(0))
\stackrel{i_*}{\longrightarrow} H_{m-n}(X_i) 
\big) 
\end{equation}
is contained in $\iota_i^*\phi_f\frown H_{m}(X_i, \partial X_i)$ for all $i$.
However, each $r$-perturbation $g$ of $f$ induces $r$-perturbations $g|_{X_i}$ of $f|_{X_i}$ and 
each set of $r$-perturbation $g_i$ of $f|_{X_i}$ induces an $r$-perturbation $g$ of $f$; therefore
$$\bigcap_{g:\,\|g-f\|\leq r} \mathrm{Im}\big(H_{m-n}(g^{-1}(0))
\stackrel{i_*}{\longrightarrow} H_{m-n}(X) 
\big)
$$
is isomorphic to the direct sum of (\ref{e:components}) over $i$ and is therefore 
a subset of $\sum_j \iota_j^*\phi_f\frown H_{m}(X_j, \partial X_j)\simeq \phi_f\frown H_m(X,\partial X)$.

In the rest of the proof, we will assume that $X$ is connected.

\heading{Poincar\'e dual of the fundamental class.} Now we will show that the Poincar\'e dual of the first obstruction equals the image of the fundamental class of the zero set of a smooth strict $r$-perturbation transverse to $0$. 
\begin{lemma}
\label{l:Poincare}
Let $X$ be a smooth oriented $m$-manifold with boundary, $A$ and $B$ be $(m-1)$-submanifolds of $\partial X$, 
$\partial X=A\cup B$, $\partial A=\partial B$, $f: (X,A)\to (\R^n, \R^{n}\setminus\{0\})$ be smooth with $0$ 
a regular value of $f$ and $f|_{\partial X}$, $[X]\in H_m(X,\partial X)$ the fundamental class of $X$, $\phi_f=f^*(\xi)\in H^n(X,A)$ 
the first obstruction and $\phi_f\frown [X]$ its Poincar\'e dual. 

Then the smooth submanifold $f^{-1}(0)$ of $X$ can be endowed with an orientation such that its fundamental class 
$[f^{-1}(0)]\in H_{m-n}(f^{-1}(0), B)$ satisfies
$$
i_*([f^{-1}(0)]) = \phi_f\frown [X]
$$
where $i: f^{-1}(0)\hookrightarrow X$ is the inclusion.
\end{lemma}
It follows immediately that $\phi_f\frown [X]$ equals to the image of the fundamental class of any smooth $g$ such that
$\phi_g=\phi_f$: this happens, in particular, if $A=|f|^{-1}(r)$ and $g$ is a smooth strict $r$-perturbation of $f$ transverse to $0$.
\begin{proof}
If $0$ is a regular value of $f$ and $f|_{\partial X}$, then $f^{-1}(0)$ is a smooth $(m-n)$-dimensional submanifold of $X$
and $\partial f^{-1}(0)=f^{-1}(0)\cap B$ is a smooth submanifold of $B$; it also follows that 
the inclusion $i:f^{-1}(0)\hookrightarrow X$ induces a homology map that maps the fundamental class 
$[f^{-1}(0)]\in H_{m-n}(f^{-1}(0), \partial f^{-1}(0))$ to $H_{m-n}(X,B)$.
Smooth manifolds can be triangulated and a triangulation of $f^{-1}(0)\cap B$ can be 
extended to a triangulation of all $B$ and subsequently to $f^{-1}(0)$, $\partial X$ and $X$~\cite[Thm. 10.6.]{Munkres:1966}. 
In the rest, we will work with simplicial (co)homology and simplicial cap product, 
which for simplicial complexes coincides with our homology theory.
We will show that the Poincar\'e dual $\phi_f\frown [X]$ of the obstruction is the image of the fundamental class $[f^{-1}(0)]$
by induction on $n$.

Let $n=1$. For the use of simplicial homology, choose an ordering of all vertices such
that the vertices in $f^{-1}[-r, 0)$ have lower rank then vertices in $f^{-1}[0,r]$. The obstruction $\phi_f$
can be represented by a simplicial cocycle $z_f$ that assigns $1$ to each oriented 1-simplex $xy$ with $f(x)<0$ and $f(y)\geq 0$,
and $0$ to other oriented 1-simplices. The fundamental class $[X]\in H_m(X, \partial X)$ consists
of all $m$-simplices in $X$ and the cap product $z_f\frown(\sum_{\Delta_m\in X^{(m)}} \Delta_m)$ consists
of all $(m-1)$-simplices $\Delta_{m-1}=[y_0,y_1,\ldots, y_{m-1}]$ in $f^{-1}(0)$ such that $[x, y_0, \ldots, y_{m-1}]$ is an $m$-simplex in $X$, $f(x)<0$ and
$f(y_j)=0$ for all $j$. The set of all such $\Delta_{m-1}$ yields a triangulation of $f^{-1}(0)$ with the orientation induced 
from the chosen orientation of $X$. It follows that $\phi_f\frown [X]$ equals the image of the fundamental class of 
$[f^{-1}(0)]$ in $H_{m-1}(X,B)$.

Let $n>1$ and $f=(f_1, f_2)$ with $f_1$ scalar valued and $f_2: X\to\R^{n-1}$. Each $x\in f^{-1}(0)$
is a regular point of $f$ and $f|_{\partial X}$, hence it is a regular point of both $f_1$, $f_1|_{\partial X}$ and $f_2$, $f_2|_{\partial X}$.
It follows that there exists a neighborhood $U$ of $f^{-1}(0)$ s.t. $0$ is a regular value of both $f_1|_U$, $f_1|_{U\cap\partial X}$ 
and $f_2|_U$, $f_2|_{U\cap\partial X}$.
Possibly changing $f_1$ and $f_2$ outside $U$ without changing $f^{-1}(0)=f_1^{-1}(0)\cap f_2^{-1}(0)$, we may
assume that $0$ is a regular value of both $f_1$, $f_1|_{\partial X}$ and $f_2$, $f_2|_{\partial X}$, 
so that $f_1^{-1}(0)$ and $f_2^{-1}(0)$
are smooth manifolds of dimensions $m-1$ and $m-n+1$, respectively.
Choose a compact $(m-1)$-submanifold $A_1$ of $A$ such that $0\notin f_1(A_1)$ and so that $A_2:=\overline{A\setminus A_1}$ 
satisfies $0\notin f_2(A_2)$. Both $A_2$ and $B\cup A_1=\overline{\partial X\setminus A_2}$ are smooth $m-1$-dimensional submanifolds
of $\partial X$, $A_2\cup (B\cup A_1)=\partial X$ and $\partial A_2=\partial (B\cup A_1)$.

The maps $f, f_1$ and $f_2$ can be considered as maps of pairs $f: (X,A)\to (\R^n, \R^{n}\setminus\{0\})$, 
$f_1: (X, A_1)\to (\R, \R\setminus\{0\})$
and $f_2: (X, A_2): (\R^{n-1}, \R^{n-1}\setminus\{0\})$. 
Let $\xi^n, \xi^1$, resp. $\xi^{n-1}$ be fundamental classes of $H^j(\R^j, \R^j\setminus\{0\})$ where $j$ equals $n$, $1$, resp. $n-1$;
here we assume a canonical orientation on $\R^j$.
Let $\phi_1:=f_1^*(\xi^1)\in H^1(X,A_1)$ and $\phi_2:=f_2^*(\xi^{n-1})\in H^{n-1}(X, A_2)$ be the corresponding obstructions.
The cross product in cohomology~\cite[p. 214]{Hatcher}
\begin{align*}
H^1(\R, \R\setminus\{0\})\times H^{n-1}(\R^{n-1}, \R^{n-1}\setminus\{0\}) \stackrel{\times}{\longrightarrow} 
H^n(\R^n, \R^n\setminus\{0\})
\end{align*}
takes $(\xi^1, \xi^{n-1})$ to $\xi^n$. Using this we obtain
$$
f^*(\xi^n)=
f^*(\xi^1\times \xi^{n-1})=(p_1 f)^* (\xi^1) \smile (p_2 f)^*(\xi^{n-1})=f_1^*(\xi^1)\smile f_2^*(\xi^{n-1})
$$
for $p_1$ and $p_2$ the projections in $B^n$ to the first, resp. the remaining components.
Comparing the left and right hand side of the last equation yields $\phi_f=f^*(\xi^n)=f_1^*(\xi_1)\smile f_2^*(\xi^{n-1})=\phi_1\smile \phi_2$.

Now we use the induction hypothesis for the $\R^{n-1}$-valued map $f_2$ and the 
subcomplexes $A_2$ and $B\cup A_1$ of $X$.
It says that $\phi_2\frown[X]$ is the image of the fundamental class $[f_2^{-1}(0)]\in H_{m-n+1}(f_2^{-1}(0), B\cup A_1)$ in
$H_{m-n+1}(X, B\cup A_1)$.

The naturality of the cap product yields the following scheme:
\begin{equation}
\label{e:diagram}
\begin{diagram}
H^1(X,A_1) & \times     & H_{m-n+1}(X, B\cup A_1) & \rTo^{\frown} & H_{m-n}(X,B) \\
\dTo^{i^*} &            & \uTo^{i_*}              &               & \uTo^{i_*} \\
H^1(f_2^{-1}(0), A_1) 
           & \,\,\times\,\,     & H_{m-n+1}(f_2^{-1}(0), B\cup A_1) 
                                                  & \rTo^{\frown} & H_{m-n}(f_2^{-1}(0),B). \\
           &            &                         &               & \uTo^{i_*} \\
           &            &                         &               &  H_{m-n}(f^{-1}(0),B) \\
\end{diagram}
\end{equation}
The pullback $\tilde{\phi}_1:=i^*\phi_1\in H^1(f_2^{-1}(0), A_1)$ is the obstruction associated to the restriction of $f_1$
to $f_2^{-1}(0)$.  The restrictions $f_1|_{f_2^{-1}(0)}$ as well as $f_1|_{\partial f_2^{-1}(0)}$ have $0$ as regular value, so using
again the induction hypothesis, we get that $\tilde{\phi}_1\frown [f_2^{-1}(0)]$ is the inclusion-induced image of 
$[f^{-1}(0)]\in H_{m-n}(f^{-1}(0), B)$ in $H_{m-n}(f_2^{-1}(0), B)$. Using
the commutativity of diagram (\ref{e:diagram}), we get that the inclusion-induced image of $[f^{-1}(0)]$
in $H_{m-n}(X,B)$ equals 
$$\phi_1\frown (\phi_2\frown [X])=(\phi_1\smile\phi_2) \frown [X]=\phi_f\frown [X]$$ 
which completes the proof.
\end{proof}
In the rest of this section, we will only need the last lemma for the case where $A=\partial X$ and $B=\emptyset$.

\heading{Reduction to the existence of a perturbation $g$ such that $g^{-1}(0)$ is connected.} Assume that $(X,A)=|f|^{-1}[0,r], |f|^{-1}(r)$,
$X$ is a smooth connected manifold, $A=\partial X$, $g$ is smooth, $0$ is a regular value of $g$, $\|g-f\|<r$, and $g^{-1}(0)$ is connected. 
The constraint $\|g-f\|<r$ immediately implies that $f$ and $g$ are homotopic as maps $A\to\R^n\setminus\{0\}$ and $\phi_f=\phi_g$.
As $X$ is connected, the group $H_{m}(X,\partial X)$ is generated by the fundamental
class of $X$ and we already know by Lemma~\ref{l:Poincare} that $\phi_f\frown [X]=\phi_g\frown [X]$ equals the image of the fundamental class of $g^{-1}(0)$
in $H_{m-n}(X)$.
But if the manifold $g^{-1}(0)$ is connected, then $H_{m-n}(g^{-1}(0))$ 
is generated by the fundamental class of $g^{-1}(0)$, so its image in $H_{m-n}(X)$ is generated by $\phi_f\frown [X]$.
It follows that $U_{m-n}(f,r)$, being the intersection over all $r$-perturbations, 
cannot contain anything else than multiples of $\phi_f\frown [X]$, and we obtain the desired inclusion
$$
\phi_f\frown H_{m-n}(X,\partial X)\supseteq U_{m-n}(f,r).
$$

So, it remains to prove that if $n+1\leq\dim X\leq 2n-3$, $X$ is connected and $A=\partial X$, 
then there exists a smooth strict $r$-perturbation $g$ of $f$ transverse to $0$ such that $g^{-1}(0)$ connected. 
To show this, we need to introduce additional concepts from differential topology.

\heading{Framed submanifolds.} Assume that $X$ is a smooth $m$-manifold endowed with a Riemannian metric; 
the results
will be independent on the choice of the metric. Let $S\subseteq X$ be a smooth $(m-n)$-submanifold contained in the
interior of $X$; for each $x\in S$, the tangent space $T_x X$ decomposes as a direct sum of the tangent space $T_x S$ 
and the normal space $N_x S$. A framing on $S$ is a trivialization of the normal bundle $NS$, 
in other words, a smooth mapping $T$ such that for each $x\in S$, $T(x)=(T_1(x),\ldots, T_n(x))$ 
is a basis of the normal space $N_x S$.

If $f: X\to\R^n$ has $0$ a regular value, then $f^{-1}(0)$ is naturally a framed $(m-n)-$submanifold, 
$T_i(x)$ being the unique vector
in $N_x f^{-1}(0)$ mapped by $df$ to $e_i\in \R^n$. We will denote these vectors by $f^*(e_i)$.
Assume that $W$ is a framed $(m-n+1)$-submanifold of $X$ with framing
$T_2(x),\ldots, T_{n}(x)$ and that $\partial W$ is the boundary of $W$.
The existence of collars implies that some neighborhood of $\partial W$ in $W$
is diffeomorphic to $\partial W\times [0,1)$ with coordinates $(w,t)$.
The framing of $W$ \emph{induces} a framing of its boundary, given by 
$(T_1(x),\ldots, T_n(x))$ where $T_1(x)$ is the vector $\partial_t$
in the ``inwards'' direction and $(T_2(x),\ldots, T_n(x))$ the framing of $W$ in $x\in\partial W$.
\begin{lemma}
\label{l:framing}
Let $X$ be a smooth $m$-manifold, $r>0$, $f: X\to\R^n$ be smooth with $0$ a regular value of $f$, $A=\partial X=|f|^{-1}(r)$, 
and $\dim X\leq 2n-3$.  Let $S\subseteq X$ be a framed boundary-free $(m-n)$-submanifold of $X$ disjoint from $A$ 
and assume that there exists a framed $(m-n+1)$-submanifold $W\subseteq X$ disjoint from $A$ so that
$\partial W=f^{-1}(0)\sqcup S$ and $W$ induces the framing of $f^{-1}(0)\sqcup S$.

Then there exists a smooth $g$ so that $\|g-f\|<r$, $0$ is a regular value of $g$ and $g^{-1}(0)=S$.
\end{lemma}
We will see that $g$ can be even chosen so that the $S$-framing $(T_1(x),\ldots, T_n(x))$ 
satisfies $T_1(x)=-g^*(e_1)$ and $T_i(x)=g^*(e_i)$ for $i>1$.

\begin{proof}
\emph{Step 1: reduction to the existence of $h$ homotopic to $f$, $h^{-1}(0)=S$.}\\
We will construct a smooth map $h$ s.t. $h^{-1}(0)=S$ and $h/|h|$ will be homotopic to $f/|f|$ as maps from $A\to S^{n-1}$. 
This is sufficient, because then we might
easily change $h$ in a collar of $A$ diffeomorphic to $A\times [0,1]$ that is disjoint from $h^{-1}(0)$
to obtain a smooth extension $e: X\to\R^n$ of $f|_A$ that coincides with $h$ outside this neighborhood.
As we have seen in the proof of Lemma~\ref{l:perturb-ext}, some positive scalar multiple
$\chi(x)\,e(x)=:g(x)$ satisfies $\|g-f\|<r$. The map $\chi$ can be chosen to be smooth:
then $0$ is a regular value of $g$ and $g^{-1}(0)=h^{-1}(0)=S$.
In the rest of the proof, we will show how to construct $h$.
\\ \\
\emph{Step 2: constructing a perturbation $\tilde{f}$ of $f$.}\\
Let $(T_2,\ldots,T_n)$ be the framing on $W$, inducing the framing $(T_1,\ldots, T_n)$ on 
$f^{-1}(0)\sqcup S = \partial W$. On $f^{-1}(0)$, $T_i$ coincides with $f^*(e_i)$.
Let $B_\epsilon^n$ be a closed neighborhood of $0\in\R^n$ consisting of regular values of $f$ and 
let $L$ be the closed straight line segment connecting $0$ and $-\epsilon e_1\in\R^n$. Then the $f$-preimage of $L$ 
is an $(m-n+1)$-submanifold of $X$ with boundary $f^{-1}(\{0,-\epsilon\}\times \{0\})$,
with a framing $(f^*(e_2),\ldots, f^*(e_n))$, where $e_2,\ldots, e_n$ are the normal vectors to $L$ for each $x\in L$.
By making $\epsilon$ possibly smaller, we can assume that $f^{-1}(L)\cap W=f^{-1}(0)$, because in a small enough neighborhood
$U$ of $f^{-1}(0)$, $f_1$ is positive on $(W\setminus f^{-1}(0))\cap U$ (by definition, $(d f_1)(T_1(x))>0$ for $x\in f^{-1}(0)$)
and negative on $(f^{-1}(L)\setminus f^{-1}(0))\cap U$.
The vector field $T_1(x)$ for $x\in f^{-1}(0)$ is in the tangent space of both $W$ and $f^{-1}(L)$; it has the inwards direction
wrt. $W$ and outward wrt. $f^{-1}(L)$.

Let $V:=f^{-1}(B_\epsilon^n)$. The restriction $f|_V$ is transverse to the closed set 
$$\R_0^- e_1:=(-\infty, 0]\times\{0\}\subseteq\R^n$$ by construction.
Using a relative version of transversality theorem, the space of smooth functions that coincide with $f$ on $V$ and are transverse
to $\R_0^- e_1$ is dense and open in $\{g\in C^\infty(X,\R^n):\,g|_V=f|_V\}$ in Whitney $C^1$-topology 
(this follows from~\cite[Thm 19.1]{Abraham:1967}) so there exists an arbitrary small perturbation $\tilde{f}$ of $f$ that is smooth, transverse to $\R_0^- e_1$ and $\tilde{f}|_V=f|_V$.
Then $\tilde{f}^{-1}(\R_0^- e_1)$ is a smooth $(m-n+1)$-submanifold of $X$ with boundary
$f^{-1}(0)=\tilde{f}^{-1}(0)$. 

The assumption $m\leq 2n-3$ implies that 
$\dim W+\dim \tilde{f}^{-1}(\R_0^- e_1)=2(m-n+1)<m$, so both $W$ and $\tilde{f}^{-1}(\R_0^- e_1)$ 
have dimension less than one half of $m=\dim X$.  Therefore, we can replace $\tilde{f}$ by another arbitrary small perturbation,
without changing it on $V$, assume that it is transverse to $\R_0^- e_1$ and moreover, $\tilde{f}(\R_0^- e_1)$
intersects $W$ only in $f^{-1}(0)$. 
Assume that $\tilde{f}$ is close enough to $f$ so that $f|_A$ is homotopic to $\tilde{f}|_{A}$ as maps from $A$ to $\R^n\setminus \{0\}$.
Without loss of generality, we may assume that $\tilde{f}^{-1}(\R_0^- e_1)$ intersects $A$ 
transversally (otherwise we replaced $\tilde{f}$ by another perturbation that differs from $\tilde{f}$ in a neighborhood of $A$)
and hence $\tilde{f}^{-1}(\R_0^- e_1)\cap A$ is an $(m-n)$ dimensional submanifold of $A$.

The submanifold $\tilde{f}^{-1}(\R_0^- e_1)$ is endowed with a framing $(\tilde{f}^*(e_2), \ldots \tilde{f}^*(e_n))$
where $e_2,\ldots, e_n$ are vectors of the canonical basis in $T_{(y,0)} \R^n$ for $y\leq 0$. This framed manifold
intersects $W$ in $f^{-1}(0)$, the tangent spaces of both manifolds coincide in $f^{-1}(0)$, $T_1$ directs inwards wrt. $W$
and outwards wrt. $\tilde{f}^{-1}(\R_0^- e_1)$, and the framing on both submanifolds coincide in $f^{-1}(0)$.
\\ \\
\begin{figure}
\begin{center}
\includegraphics[scale=1.1]{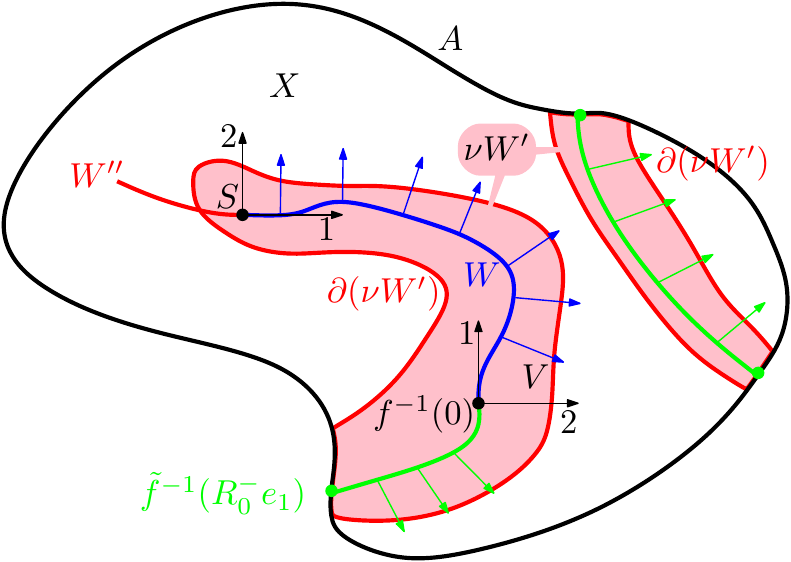}
\end{center}
\end{figure}
\emph{Step 3: gluing $W$ and $\tilde{f}^{-1}(\R_0^-\,e_1)$ to one smooth submanifold.}\\
Both submanifolds $W$ and $\tilde{f}^{-1}(\R_0^- e_1)$ of $X$ intersect in their common boundary $f^{-1}(0)$ and both 
the tangent spaces and framings coincide in $f^{-1}(0)$.
We would like to smoothly ``glue'' them to one framed manifold 
$W\cup \tilde{f}^{-1}(\R_0^- e_1)$ but unfortunately, such union does not need to yield a \emph{smooth} submanifold in general.

We claim that there exists a smooth framed manifold $W'$ that coincides
with $W\cup \tilde{f}^{-1}(\R_0^- e_1)$ everywhere except on a neighborhood of $f^{-1}(0)$ in $X$ that can be chosen to be arbitrary small. 
Choose a continuous tangent vector fields $v$ in $W\cup \tilde{f}^{-1}(\R_0^- e_1)$ that is smooth on $W$ 
and smooth on $\tilde{f}^{-1}(\R_0^- e_1)$,
such that $v|_{f^{-1}(0)}$ is nonzero and points inwards wrt. $W$ and outwards wrt. $\tilde{f}^{-1}(\R_0^- e_1)$
(it may coincide with $f^*(e_1)$).
The flow of this vector fields induces collar neighborhoods $C_1$ resp. $C_2$ of $f^{-1}(0)$ in $W$ resp. $\tilde{f}^{-1}(\R_0^- e_1)$ 
contained in $X$ diffeomorphic to $f^{-1}(0)\times [0,\epsilon)$, resp.  $f^{-1}(0)\times (-\epsilon,0]$ for some $\epsilon>0$.
Let as denote the embeddings $f^{-1}(0)\times [0,\epsilon)\to X$ and $f^{-1}(0)\times (-\epsilon, 0]\to X$ by $w_1$ and $w_2$, respectively.
Let $w: f^{-1}(0)\times (-\epsilon,\epsilon)\to X$ be defined by $w_1(x,t)$ for $t\geq 0$ and $w_2(x,t)$ for $t<0$:
this map is a $C^1$-embedding (the differentials $dw_1$ and $dw_2$ coincide on $f^{-1}(0)$) and is $C^\infty$ whenever $t\neq 0$.

Let $\psi_\alpha: U_\alpha\to\R^m$ be a collection of $X$-charts and let $\{V_\beta\}_\beta$ be an open covering of 
$f^{-1}(0)\times [-\epsilon/2, \epsilon/2]$
such that for each $V_\beta$, $w(\overline{V_{\beta}})$ is contained in some $U_{\alpha(\beta)}$. Further, let $V_\beta'\subseteq V_\beta$
be so that $\overline{V_\beta'}\subseteq V_\beta$ and $\{V_\beta'\}_\beta$ is still an open covering of $f^{-1}(0)\times [-\epsilon/2, \epsilon/2]$.
Let $\beta_1,\ldots, \beta_k$ be such that
$\cup_j V_{\beta_j}'$ is an open neighborhood of $f^{-1}(0)\times \{0\}$. 
The space $C^\infty(\overline{V_{\beta_1}}, \R^m)$ is dense in $C^1(\overline{V_{\beta_1}}, \R^m)$
with the Whitney $C^1$-topology \cite[Thm. 2.4]{Hirsch:76} so we may replace 
$\psi_{\alpha(\beta_1)}\circ w|_{\overline{V_{\beta_1}}}: \overline{V_{\beta_1}}\to\R^m$
by an arbitrary close map (in the Whitney topology) $\overline{V_{\beta_1}}\to\R^m$ that is smooth on $V_{\beta_1}'$ and unchanged in a neighborhood of 
$\partial V_{\beta_1}$. This defines a map $w_1': f^{-1}(0)\times [-\epsilon/2,\epsilon/2]\to X$ that is smooth on $V_{\beta_1}'$ 
and coincides with $w$ on $f^{-1}(0)\times \{\pm\epsilon/2\}$. On $w_1'(\overline{V_{\beta_1}})$, we can also use the $\psi_{\alpha(\beta_1)}$-chart to define 
a~framing that is smooth on $w_1(V_{\beta_1}')$ and coincides with the original framing on a neighborhood of $w_1'(\partial V_{\beta_1})$.
In the same way, we smoothen the function on $V_{\beta_2}',\ldots, V_{\beta_k}'$
and obtain a smooth map $w': f^{-1}(0)\times (-\epsilon, \epsilon)\to X$ arbitrary close to $w$ that coincides with $w$ on 
a neighborhood of $f^{-1}(0)\times \{\pm\epsilon\}$. If we chose $w'$ close enough to $w$, it is an embedding by~\cite[Thm 1.4]{Hirsch:76}.
The manifold $W'$ can now be defined as
$$
W':=\mathrm{Im}(w')\cup \big((W\cup \tilde{f}^{-1}(\R_0^- e_1))\setminus \mathrm{Im}(w')\big)
$$
which is a smooth embedded framed submanifold that coincides with $W\cup \tilde{f}^{-1}(\R_0^- e_1)$ 
except on a neighborhood of $f^{-1}(0)$ that can be chosen to be arbitrary small.  
The framing coincides with the framing on $W\cup \tilde{f}^{-1}(\R_0^- e_1)$ outside $\mathrm{Im}(w')$.
By construction, the boundary of $W'$ consists of $S$ and a submanifold $W'\cap A$ of $A$.
\\ \\
\emph{Step 4: choice of the metric.}\\
Let us choose a vector field $v$ on $A=\partial X$ in the inwards-direction so that for $x\in W'\cap A$, $v(x)\in T_x W'$. 
This can be extended to a nowhere zero vector field in a neighborhood of 
$A$ and used to define a collar neighborhood of $A$ diffeomorphic to $A\times [0,\epsilon)$ for some $\epsilon>0$, 
the diffeomorphism induced by the flow of $v$.  We endow $X$ with a new smooth Riemannian metric that is a product metric on this neighborhood.
Due to this choice of the metric, the geodesics in $A$ coincide with the geodesics in $X$. 
In what follow, we will assume that such metric has been defined and we identify the given framing of $W'$ and $S$ with normal vectors
wrt. this metric. In particular, $W'$ intersects $A$ orthogonally and the $W'$-framing vectors in $W'\cap A$ 
are all in the tangent space of $A$.
\\ \\
\emph{Step 5: construction of $h$.}\\
The first framing vector $T_1$ of the $S$-framing can be extended to a smooth tangent (wrt. $W$) vector field 
in a neighborhood of $S$ in $W$ and further to a neighborhood of $S$ in $X$. 
The flow of $-T_1$ then generates an external collar $C$ of $W'$ diffeomorphic to $S\times [0,\epsilon]$ for some $\epsilon>0$
such that $W'':=C\cup W'$ is a smooth submanifold of~$X$.  Using charts and partition of unity, we may easily extend the $W'$-framing
to a framing on $W''$.
Without loss of generality, we can assume that the external collar $C$ of $W$ is disjoint from $A$.
The flow of $-T_1$ induces a neighborhood $\nu(S)$ of $S$ in $W''$ diffeomorphic to $S\times [-\epsilon,\epsilon]$ with 
$S\times (-\epsilon,0]$ corresponding to a neighborhood of $S$ in $W'$ and $S\times [0,\epsilon)$ to the open external collar of $S$ 
contained in $W''$. The projection on $[-\epsilon, \epsilon]$ defines a smooth scalar valued map $h_1: \nu(S)\to\R$ s.t. 
$h_1^{-1}(0)=S$, $h_1\geq 0$ on $C$ and $h_1\leq 0$ on $\nu(S)\cap W'$: $(h_1)_*$ maps $T_1$ (which directs inwards to $W$) 
to $(-\partial_x)\in T_0\R$. 

By construction, $\tilde{f}_1$ is negative in $W''\cap A$.  
Let $U$ be a closed neighborhood of $A$ in $X$ such that $\tilde{f}_1$ is still negative on $U\cap W''$
and extend $h_1$ to a smooth map $W''\to\R$ such that $(h_1)|_{U\cap W''}=(\tilde{f}_1)|_{U\cap W''}$
and $h_1<0$ on $W'\setminus S$.

The geodesic flow of the $W''$-framing of $W''$ induces a diffeomorphism $\varphi$ of $W''\times B^{n-1}$ and some set $\nu(W'')\subseteq X$,
where $B^{n-1}$ is the closed ball in $\R^{n-1}$ of small enough diameter (due to our choice of the metric, framing vectors in $W'\cap A$
induce geodesics in $A$). 
This set $\nu(W'')$ is not open in $X$, but it contains an open neighborhood of $W'$.
The projection on $B^{n-1}$ defines a~smooth function $h': \nu(W'')\to\R^{n-1}$ transverse to $0$ 
such that $h'^{-1}(0)=W''$ and $h'$ induces the given framing on $W''$. 
Let us extend $h_1$ to a smooth scalar valued map $\nu (W'')\to\R$ arbitrarily and finally define $h: \nu (W'')\to\R^n$ by $h=(h_1, h')$. 
It is easy to see that $h^{-1}(0)=S$ and $0$ is a regular value of $h$.
Summarizing the construction, we have a (closed) neighborhood $\nu (W'')$ of $W'$ in $X$ and a smooth map $h: \nu (W'')\to\R^n$ such that
\begin{itemize}
\item $h^{-1}(0)=S$, the original framing of $S$ equals $(-h^*(e_1), h^*(e_2),\ldots h^*(e_n))$,
\item $(h_2,\ldots, h_n)^{-1}(0)=W''$ and $(h_2,\ldots, h_n)$ induces the original framing on $W'\subseteq W''$,
\item $W''\cap h^{-1}(\R_0^- e_1)=W'$,
\item $h|_{U\cap W''}=\tilde{f}|_{U\cap W''}$.
\end{itemize}
By construction, $h$ restricted to $\partial (\nu(W''))$  has values in $\R^n\setminus\{\R_0^-\,e_1\}$:
this is because $h(x)\in\R_0^- e_1$ implies $x\in W'$ and $W'$ is in the interior of $\nu (W')$. 
The topological space $\R^n\setminus(\R_0^- e_1)$ is homotopically trivial 
as it deformation retracts to a point, so $h|_{\partial\nu(W'')}$ can be extended to a continuous map 
$\overline{X\setminus \nu(W'')}\to\R^n\setminus\{\R_0^-\,e_1\}$ 
(for this, we need the homotopy extension property of $\overline{X\setminus \nu(W'')}$ and its closed subset $\partial \nu(W'')$)
which in turn defines an extension $h: X\to\R^n$ of the map $h|_{\nu(W'')}$ that we have already defined.
Possibly perturbating $h$ slightly outside of some neighborhood of $W''$, we may assume that it is smooth~\cite[Thm. 2.5]{Kosinski:2007}.
By construction, $0$ is a regular value of $h$.
\\ \\
\emph{Step 6: the restriction $h|_A$ is homotopic to $f|_A$.}\\
We will show that $H:=h/|h|$ and $F:=f/|f|$ are homotopic as maps from $A\to S^{n-1}$. 
Let $\tilde{F}$ be the restriction of $(\tilde{f}/|\tilde{f}|)$ to the open neighborhood $U$ of $A$.
We assumed that $f|_{A}$ is homotopic to $\tilde{f}|_{A}$ as maps $A\to\R^n\setminus\{0\}$,
so the sphere-valued maps $F|_{A}$ is homotopic to $\tilde{F}|_{A}$
and it remains to show that $\tilde{F}|_{A}$ is homotopic to $H|_{A}$.

By construction, $(\tilde{f}|_{U})^{-1}(\R_0^- e_1)=({h|_U})^{-1}(\R_0^- e_1)=U\cap W''$, both maps are nowhere zero on $U$,
they coincide on $U\cap W''$ and both maps induce the same framing on $U\cap W''$. It follows that 
$\tilde{F}^{-1}(-e_1)=(H|_U)^{-1}(-e_1)=U\cap W''$ and both $\tilde{F}$ and $H$ induce the same framing on $U\cap W''$ (it coincides
with the original framing of $U\cap W''$ up to scalar multiples of framing vectors): 
this framing restricts on $W''\cap A$ to a framing of the normal space to $W''\cap A$ in $A$. 
Therefore, for $x\in \tilde{F}^{-1}(-e_1)\cap A$, $\tilde{F}_*(x)=H_*(x)$ and consequently $(\tilde{F}|_{A})_*(x)=(H|_{A})_*(x)$.
It follows that $\tilde{F}|_{A}$ and $H|_{A}$ induce the same framing of
the normal bundle $N((\tilde{F}|_{A})^{-1}(-e_1)\cap A)$ in $A$ and by~\cite[Lemma 4, p. 48]{Milnor:97}, 
$\tilde{F}|_{A}\sim H|_{A}$ are homotopic.
\end{proof}

\heading{Connecting disconnected components.}
In this section, we show that if $S_1$ is a framed submanifold of $X$ with dimension at least 1 and codimension at least 3,
then there exists a framed submanifold
$W\subseteq X$ such that $\partial W=S_1\sqcup S_2$ where $S_2$ is connected. 
This will finish the proof of Theorem~\ref{t:Poincare}, because it follows that for the framed submanifold 
$S_1:=f^{-1}(0)$, we can construct a strict $r$-perturbation $g$ of $f$ s.t. $g^{-1}(0)=S_2$ is connected by Lemma~\ref{l:framing}.
The constraint $n+1\leq m\leq 2n-3$ that we assume in Theorem~\ref{t:Poincare} implies that the dimension of $f^{-1}(0)$
is at least $1$ and that $n\geq 4$, so all the dimensional 
assumptions of Lemma~\ref{l:connected_sum} and Lemma~\ref{l:framing} are satisfied.

\begin{lemma}
\label{l:connected_sum}
Let $X$ be a smooth connected manifold, $S_1$ a framed closed submanifold of $X$\footnote{That is, $S_1$ is compact and without boundary.} and assume that
$1 \leq \dim S_1\leq \dim X-3$.
Then there exists a framed submanifold $W$ in the interior of $X$ such that 
$\partial W=S_1\sqcup S_2$, $W$ induces the framing on $S_1\sqcup S_2$ and $S_2$ is connected.
\end{lemma}

The main idea of the proof is to construct a manifold $W_1\simeq S_1\times [0,1]$, cut out two holes in $S_1\times \{1\}$
around $x$ and $y$ that are in different components of $S_1\times \{1\}$ and connect them with a tubular $(\dim S_1+1)$-dimensional neighborhood of a curve connecting $x$ and $y$. While there 
is a well-known construction called ``boundary connected sum'' for abstract differential manifolds~\cite{Kosinski:2007}, we could not
find any reference that this can be done all inside the ambient space $X$, so here we present the sketch of our construction. 

\begin{proof}
Let $m:=\dim X$ and $n:=\dim X-\dim S_1$. 
By the product neighborhood theorem, the framing of $S_1$ determines a diffeomorphism 
$d$ from $S_1\times B_2^n$ to a (closed) neighborhood $U$ of $S_1$ in $X$, where $B_2^n\subseteq\R^n$ is the closed ball of diameter $2$.
Let as choose a smooth metric on $S_1$, extend it to a product metric on $U$ via the diffeomorphism $d$ and smoothly extend it 
to a metric on the whole of $X$. Let $v_1,\ldots, v_n$ be the vector fields on $U$ defined as the $d_*$-image of the euclidean coordinate 
vector fields $e_1,\ldots, e_n\in TB_2^n$ identified with vectors of $T(S_1\times B_2^n)\simeq TS_1\times TB_2^n$ 
orthogonal to $TS_1\times \{0\}\subseteq T_{(s,b)} (S_1\times B_2^n)$.
The image $d(S_1 \times [0,1]\times \{0\})=:W_1$ is a smooth submanifold of $X$ of dimension $m-n+1$ contained in $U$,
with boundary $\partial W_1=S_1\sqcup S_1'$ where $S_1'=d(S_1\times \{1\}\times \{0\})$. The vectors $v_2,\ldots, v_n$ form a framing
of $W_1$ and the vector field $(v_1)|_{W_1}$ is tangent to $W_1$ such that on the boundary, $v_1|_{S_1}$ goes in the ``inwards'' 
and $v_1|_{S_1'}$ in the ``outwards'' direction wrt. $W_1$.  Further, $S_1'$ is a~diffeomorphic copy of $S_1$, so it has the same 
number of connected components. Let $x,y\in S_1'$ be two points in different components of $S_1'$. 

Let $\varphi: [0,1]\to X$ be a smooth embedded curve such that $\varphi(0)=x$, $\varphi(1)=y$, $\varphi(t)\notin W_1$ for $t\neq 0,1$ and
there exists an $\delta>0$ such that $\dot{\varphi}(t)=v_1(\varphi(t))$ for $t\in [0,\delta]$ and 
$\dot{\varphi}(t)=-v_1(\varphi(t))$ for $t\in [1-\delta,1]$. 
Such curve exists, because $X$ is connected and the codimension of $W_1$ in $X$ is at least two, so $X\setminus W_1$ is still connected.
Without loss of generality, we may assume that $\delta$ is small enough so that $\varphi([0,\delta]\cup [1-\delta, 1])$ 
is in the image of $d: S_1\times B_2^n\to U$.

The curve $\varphi[0,1]$ is contractible, so any fibre bundle over it is trivial and admits a global section; in particular,
there exists a global section $T$ of the principal bundle of all framings of its $m-1$-dimensional normal bundle. 
That means, for $x=\varphi(t)$, $T(x):=(u_1,\ldots, u_{m-1})$ is a framing of $N_x \varphi[0,1]$ and $T$ is smooth. 
Any other framing of $\varphi[0,1]$ is determined by a smooth map $[0,1]\to Gl(m-1)$ which linearly transforms the framing
vectors in each $\varphi(t)$. Let $w_1,\ldots,w_{m-n}$ be a basis of $T_x S_1'$, resp.  $T_y S_1'$, oriented so that
$(w_1(x),\ldots, w_{m-n}(x), v_2(x),\ldots, v_n(x))$ has the same orientation of $N_x \varphi[0,1]$ as $T(x)$ and 
$(w_1(y),\ldots, w_{m-n}(y), v_2(y),\ldots, v_n(y))$ has the same orientation of $N_y \varphi[0,1]$ as $T(y)$. 
Let as naturally extend the vector fields $w_1,\ldots, w_{m-n}$ to $\varphi([0,\delta]\cup [1-\delta, 1])$ by
means of parallel transport along $\varphi$.  The connectedness
of $Gl^+(m-1)$ implies that there exists a framing $T'$ of $N(\varphi [0,1])$ such that 
$T'(\varphi(t))$ coincides with $(w_1,\ldots, w_{m-n}, v_2,\ldots, v_n)$ for $t\in [0,\delta]\cup [1-\delta, 1]$.
By a slight abuse of notation, we again denote the first $m-n$ framing vectors of $T'$ by $w_1,\ldots, w_{m-n}$: these
are normal vector fields on $\varphi[0,1]$ extending the already defined $\{w_i\}_i$ in $\varphi([0,\delta]\cup [1-\delta, 1])$.

For any $a\in \varphi[0,1]$ and $u\in\R^{m-n}$, let $F(a,u)$ be equal to $\gamma(1)$ where $\gamma$ is a geodesic, $\gamma(0)=a$
and $\dot\gamma(0)=u_1 w_1+u_2 w_2+...+u_{m-n} w_{m-n}$ whenever the geodesic is defined on $[0,1]$. If $\epsilon>0$ is small enough, then
$F: \varphi[0,1]\times B_\epsilon^{m-n}\to X$ is a smooth embedding and its image is an embedded $(m-n+1)$-dimensional submanifold
of $X$ (with corners in $\{x,y\}\times \partial B_\epsilon^{m-n}$). Using the properties of our metric, the $S_1'$-geodesics in $x$, 
resp. $y$ coincide with the geodesics in $X$, so $F$
maps $\{x,y\}\times B_\epsilon^{m-n}$ to a closed neighborhood $D_x\sqcup D_y$ of $\{x, y\}$ in $S_1'$, where $D_x$ resp. $D_y$
is a geodesic $\epsilon$-ball in $S_1'$.  

If $t<\delta$ or $t>1-\delta$, then $F(\varphi(t),B_{\epsilon}^{m-n})$ is contained in $U$ and disjoint from $W_1$ due to the choice of the 
product metric. If $t\in [\delta,1-\delta]$, then there exists some $\epsilon(t)<\epsilon$ and a neighborhood $U(t)$ of $t$ 
such that $F(\varphi(U(t)),B_{\epsilon(t)}^{m-n})$ is disjoint from $W_1$. By compactness of $[0,1]$, we can make $\epsilon$ smaller
and assume that $F(\varphi(0,1)\times B_\epsilon^{m-n})$ is disjoint from $W_1$.
\begin{figure}
\begin{center}
\includegraphics{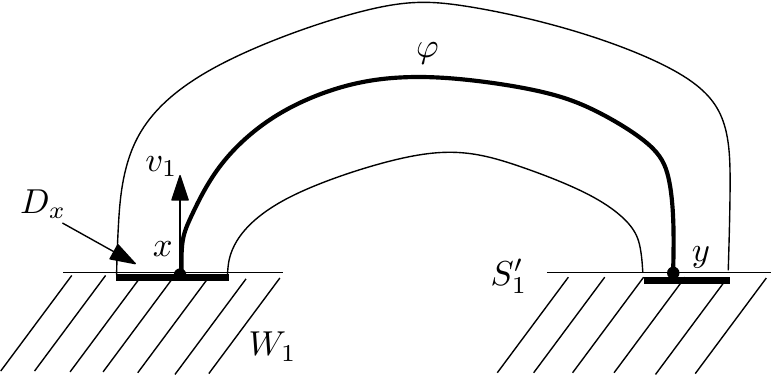}
\end{center}
\caption{The $m-n$-dimensional neighborhoods $D_x$ of $x$ and $D_y$ of $y$ in $S_1'$ are connected via an $m-n+1$ dimensional
tubular neighborhood $T$ of a curve $\varphi$ connecting $x$ and $y$.}
\label{f:connected}
\end{figure}
Let $T:=F(\varphi[0,1]\times B_\epsilon^{m-n})$. This is a smooth contractible $(m-n+1)$-manifold (with corners), 
so its normal bundle admits a global framing. By an argument completely analogous to the one above, we may extend the framing 
$v_2,\ldots, v_n$ defined on $F(\varphi([0,\delta]\cup [1-\delta, 1])\times B_\epsilon^{m-n})$ to a smooth framing of $T$.

By construction, $T\cap W_1$ consists of two $(m-n)$-discs $D_x=F(\{x\}\times B_\epsilon^{m-n})$ and $D_y$ in $S_1'$.
At this point, $W_1\cup T$ is a framed manifold, but we still need to ``smooth the corners'' $\partial D_x$ and $\partial D_y$.

Let $\psi: [0,1]\to [\frac{1}{2}\epsilon, \epsilon]$ be a smooth function such that $\psi(0)=\psi(1)=\epsilon$, 
$\psi'(0)=-\infty$, $\psi'(1)=\infty$ and further, for some $\beta>0$, the inverse function 
$(\psi|_{[0,\beta]})^{-1}: [\psi(\beta),\epsilon]\to [0,\beta]$ can be extended to a \emph{smooth} function 
$[\psi(\beta), \infty)\to [0,\beta]$ by sending each $x>\epsilon$ to $0$, and similarly 
$(\psi|_{[1-\beta,1]})^{-1}: [\psi(1-\beta),\epsilon]\to [1-\beta,1]$ can be extended to a \emph{smooth} function 
defined on $[\psi(1-\beta),\infty)$ by mapping each $x>\epsilon$ to $1$.\footnote{
Equivalently, the graph of $\psi$ united with $\{0,1\}\times [1,\infty)$ is a smooth submanifold of $\R^2$.
}
Finally, define $T'\subseteq T$ by 
$$
T':=\{F(\varphi(t),u):\,|u|\leq \psi(t)\}.
$$
We claim that $W:=W_1\cup T'$ is a smooth manifold with boundary. By construction, $W_1$ and $T'\setminus W_1$ are smooth manifolds
with $W_1\cap T'=D_x\cup D_y$, so we just need to analyze their intersection. 

Let $u$ be in the interior of $D_x$, resp. $D_y$. Let $V$ be an open neighborhood of $v$ with positive distance from $\partial D_x$ 
(resp. $\partial D_y$) and define an $S_1'$-chart $\phi_u: V\to\R^{m-n}$ that maps a neighborhood $V\subseteq D_x$
(resp. $D_y$) of $u$ to $\R^{m-n}$. Let $N\subseteq X$ be a neighborhood of $v$ in $X$ that is disjoint from the topological boundary
of $W$ in $X$ and $N\cap S_1'\subseteq V$. 
Then the diffeomorphism $(\phi_u, \mathrm{id})\circ d^{-1}$ takes $N$ to an open subset of $\R^m$ such that $N\cap W$ is the preimage of
$\R^{m-n+1}\times \{0\}$ (the projection to the $(m-n+1)$'th component of the image of $N$ is a neighborhood of $1\in\R$).

It remains to show that any $v\in\partial D_x$ resp. $\partial D_y$ is in the boundary of $W=T'\cup W_1$, that is, some neighborhood 
of $v$ in $W$ is mapped by an $X$-chart to $\R^{m-n}\times (-\infty,1]\times \{0\}$. 
Let $\phi_v$ be an $S_1'$-chart mapping a neighborhood $V$ of $x$ in $S_1'$ to $\R^{m-n}$ such that $u_v(v)=0$ and let 
$N:=d^{-1}(V\times (0,1+\delta))\subseteq X$ be a neighborhood of $v$ in $X$. 
Let $(\phi_v, \text{id})\circ d^{-1}: N\to\R^m$ be an $X$-chart: it maps
\begin{itemize}
\item $S_1'\cap N$ to $\R^{m-n}\times \{1\}\times \{0\}$,
\item $W_1\cap N$  to $\R^{m-n}\times (0, 1]\times\{0\}$, 
\item $T'\cap N$ to $\{(z_1,\ldots, z_{m-n}, y, 0,\ldots, 0):\,\,1\leq y\leq \omega(z)$\}, and
\item $W\cap N$ to $\{(z_1,\ldots, z_{m-n}, y, 0,\ldots, 0):\,\,0\leq y\leq \omega(z)\}$
\end{itemize}
where $\omega(x)=1$ whenever $x\notin D_x$ and $\omega(z)=\sup \,\{t+1:\,d(z,t,0)\in T'\}$.

If $e(u)$ is the geodesic distance of $u$ from $x$, then $\omega(z)=\sup\,\{t+1:\,e(z)\leq \psi(t)\}$.
The function $\psi$ maps a neighborhood of $0$ to some $(\epsilon-\alpha, \epsilon]$ and the inverse function
$\psi^{-1}$ to this restriction has derivative $0$ in $\epsilon$. If we extend $\psi^{-1}(a)$ to be $0$ for $a>\epsilon$,
we get a smooth function from a neighborhood of $\epsilon$ in $\R$ to nonnegative numbers with 
$\psi^{-1}(\epsilon)=(\psi^{-1})'(\epsilon)=0$.
We can rewrite $\omega(z)$ to $\psi^{-1}(e(z))+1$ to see that it is a~smooth function defined on a neighborhood of $v$ in $S_1'$
with zero gradient in $v$ (note that $e(v)=\epsilon$). 
It follows that $W\cap N$ is diffeomorphic to some neighborhood of $(0,\ldots, 0,1,0,\ldots, 0)$ in
$\{(z_1,\ldots, z_{m-n}, y, 0,\ldots, 0):\,\,y\leq \omega(z)\}$ which is diffeomorphic to $\R^{m-n}\times (-\infty, 0]\times \{0\}$.
\begin{figure}
\begin{center}
\includegraphics{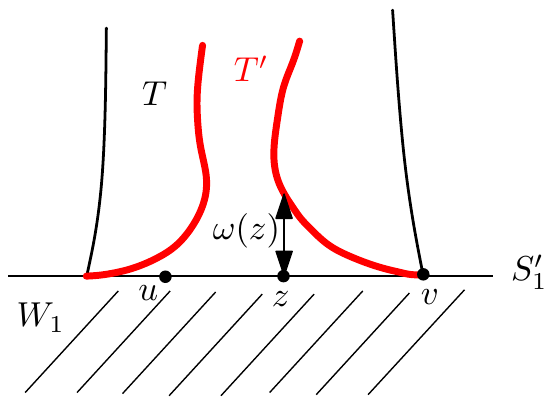}
\end{center}
\caption{Illustration of corner smoothing. $T'\subseteq T$ is chosen so that, in a neighborhood of $v\in\partial D_x$, 
$T'$ consists of points in $W_1\cup T$ whose distance from $W_1$ is bounded by a function $\omega(z)$ that has zero gradient in $v$.}
\label{f:smoothing}
\end{figure}

Hence $W$ is a smooth framed submanifold of $X$ of dimension $m-n+1$, the framing of the normal bundle being a restriction
of the framing of $W_1$ and $T$. If $m-n\geq 2$, then $\dim \partial D_x\simeq S^{m-n-1}$ is a sphere of dimension at least one
and hence is connected: by construction,  $\partial D_x$ and $\partial D_y$ are in the same component of the boundary $\partial W$.
If $m-n=1$, then $S_1'$ is a finite disjoint union of circles and yields $\partial W$ 
to be a connected sum of two circles containing $x$ and $y$ united with $S_1$ and other components of $S_1'$. 
In both cases, the number of connected components of $\partial W\setminus S_1$ is smaller then the number of connected components of $S_1'$.

The compactness of $X$ and $S_1$ implies that $S_1$ has only a finite number of components. 
In the same way as above, we may continue attaching tubular neighborhoods of curves connecting different components of $S_1'$
to obtain a framed manifold $W$ such that $\partial W\setminus S_1=:S_2$ is already connected.
\end{proof}
\section{Proof of Theorem~\ref{t:triviality}}\label{a:triv}
We need the following  observation, cf. \cite[p. 9]{Bryant-handbook}.
\begin{lemma}\label{l:dual}
Let $(X,A)$ be an $m$-dimensional pair of simplicial complexes such that $A$ contains the $(i-1)$-skeleton of $X$. Then there is an $(m-i)$-subcomplex $Y$ of a subdivision $X'$ of $X$ such that $Y$ is disjoint with $A$ and $X'\subseteq A*Y$.

Moreover, for each $(m-i)$-simplex $\tau$ in $Y$ there is an $i$-simplex
$\s$ of $X\setminus A$ such that the following holds: $\s'*\t$
is a simplex of $X'$ if and only if $\s'$ is a face of $\s$.\end{lemma}
The subcomplex $Y$ from the previous lemma is called the \emph{dual complex to $A$ in $X$}.
Also the simplex $\s$ is called the \emph{dual simplex to $\t$.}
\begin{proof}
We proceed by induction on the dimension of $X$. 
When $m=i-1$, the statement holds for the empty subcomplex $Y$ of $X$.

When $m\ge i$ we can use the induction hypothesis on $X^{(m-1)}$ to obtain the dual complex $Y^{(m-i-1)}\subseteq X'^{(m-1)}$ where $X'^{(m-1)}$ is a subdivision of $X^{(m-1)}$. For each $m$-simplex $\s$ of $X\setminus A$ we have that $\s\cong \partial\s * b_\s$ where $b_\s$ is the barycenter of $\s$. Since $\partial \s$ can be considered as a subcomplex of $X'^{(m-i-1)}$, we have that $X^{(m-1)}\cup \s$ is a subspace of $X'^{(m-1)} * b_\s$. Similarly whole $X$ is a subspace of $X'^{(m-i)}* \{b_\s \: \s\text{ is an $m$-simplex of }X\}$. Therefore the desired $Y$ can be set to $Y^{(m-i-1)} * \{b_\s\: \s\text{ is an $m$-simplex of }X\}$.

To prove the second statement, we need to distinguish two cases. First, when $m=i$, then the dual simplex to each $[b_\s]$ is $\s$. Second, when $m>i$, each $(m-i)$-simplex $\t$ has the form $b *\t'$ for some $(m-i-1)$-simplex $\t'$ of $Y^{(m-i-1)}$. The dual simplex $\s$ to $\t$ is equal to the dual simplex of $\t'$ from the induction since $\s' *\t\in X'$ if and only if $\s' * \t'\in X'^{(m-1)}$.
\end{proof}

Observation~\ref{o:triv} directly follows from the following lemma:
\begin{lemma}\label{l:triv}
Let $f\:K\to\R^n$ be a map such that $(X,A):=|f|^{-1}([0,r],
\{r\})$ is pair of simplicial complexes. If the map $f|_A$ can
be extended to a map $f^{(i-1)}\:X^{(i-1)}\cup A\to S^{n-1}$,
then there is an $r$-perturbation $g$ of $f$ such that $g^{-1}(0)$
is an
$(m-i)$-dimensional subcomplex of some subdivision of $X$.\end{lemma}

\begin{proof}
To prove the lemma we need to find an extension $g\:X\to\R^n$ of a given map $f^{(i-1)}\:X^{(i-1)}\cup A\to S^{n-1}$ such that $g^{-1}(0)$ is a simplicial complex of dimension at most $m-i$. Let $Y$ be the dual complex to $X^{(i-1)}\cup A$ in $X$. We define the extension $g\:X\to\R^n$ by $g:=f*\boldsymbol{0}_Y$, that is, for each point $(x,t,y)$ of each simplex $\s*\t$ in $X\subseteq (X^{(i-1)}\cup A)* Y$ we define $g(x,t,y):=tf(x)+(1-t)0$. Clearly, $g^{-1}(0)=Y$.
\end{proof}

\begin{proof}[Proof of Theorem~\ref{t:triviality}] We are given an $m$-dimensional simplicial pair $(X,A)$, natural numbers $i$ and $n$ such that $n<i\le (m+n)/2-1$ and a map $h\:X^{(i-1)}\cup A\to S^{n-1}$. 
To prove the theorem, we will find an extension $g\:X\to\R^n$ of $h$ such that $g^{-1}(0)$ is a cell complex obtained from an $(m-i-1)$-dimensional simplicial complex by attaching cells of dimension $m-n$.

We will use the dual complex $Y$ to $X^{(i-1)}\cup A$ in
$X$ as in Lemma~\ref{l:dual}. Part of the map $g$ is easy to define, namely, for every $(x,t,y) \in (X^{(i-1)}\cup A)*Y^{(m-i-1)}$ we set $g(x,t,y):=(1-t) h(x)$.
(Note that for $t=1$, we have $(x,t,y)\in Y^{(m-i-1)}$.) For the rest, we need to define $g$ on each $ \Delta:=(\partial \s) * \t$ where $\t$ is an arbitrary $(m-i)$-simplex of $Y$ and $\s$ is its dual $i$-simplex in $X\setminus A$. We have that $\Delta=\partial\s*(b_\t* \partial\t)=(\partial\s*b_\t)*\partial\t$ where $b_\t$ is the barycenter of $\t$. Therefore we can write each point $p$ of the join $( \partial\s*b_\t)*\partial\t$ as $p=(x,s,t,y)$ where $x\in\partial \s$, $y\in\partial\t$ and $s,t\in[0,1]$. (We have $p\in\partial\s$ for $s=0$ and $p\in\t$ for $t=1.)$
\begin{enumerate}
\item \label{i:exten-kill-homology}
In the case where $h|_{\partial\s}$ is homotopically nontrivial, we define 
$$ g(x,s,t,y):=(1-t)\big((1-s)h(x)-{ts} \iota(y)\big),$$ where $\iota\:\partial\t\cong S^{m-i-1}\to S^{n-1}$ is an embedding of $\partial\t$ to the equatorial $(m-i-1)$-subsphere. Here we need that $m-i\le n$.
\item Otherwise, we choose an arbitrary extension $h'\:(\partial\s)*b_\t\to S^{n-1}\subseteq\R^n$ of $h|_{\partial\s}$ and define 
$$g(x,s,t,y):=(1-t) h'(x,s).$$

We can see that ${(g|_{\Delta}})^{-1}(0)$
is equal to $\partial\t$ in this case.
\end{enumerate}
To finish the proof, it suffices to show that in the case~\ref{i:exten-kill-homology} above, ${(g|_{\Delta}})^{-1}(0)\cong\Cone (\eta)$ for some $\eta\:S^{m-n-1}\to \partial\t\cong S^{m-i-1}$. Roughly speaking, we need to solve the equation $h|_{\partial\s}(x)=\iota(y)$, that is, to identify the $(h|_{\partial\s})$-preimage of equatorial $(m-i-1)$-subsphere of $S^{n-1}$. Informally, our strategy is to employ the fact that the elements $[h|_{\partial\s}]$ of the stable homotopy group $\pi_{i-1}(S^{n-1})$ are  iterated suspensions and thus, without loss of generality, the $(h|_ {\partial\s})$-preimage of the equatorial $(m-i-1)$-subsphere  is the equatorial subsphere of the same codimension (and the map  $\eta$ above is the restriction of  $h$ onto this subsphere).

\begin{itemize}
\item Formally, by the Freudenthal suspension theorem we know that $[h|_{\partial \s}]$ equals a $j$-fold suspension $\Sigma^j [\eta]$ for some
$\eta\:S^{i-1-j}\to S^{n-1-j}$ assuming the condition $i-1-j\le 2(n-1-j)-1$. Given the requirement $n-1-j=m-i-1$, the condition is equivalent to $i\le (m+n)/2-1$---the assumption of the theorem. \item Without loss of generality, we can assume that  $h|_{\partial \s}=\Sigma^j\eta$. Indeed, in general, there is a homotopy $H\:h|_{ \partial\s}\sim \Sigma^j \eta$. We can parameterize a regular neighborhood $N$ of $\partial\s$ in $(\partial\s*b_\t)$ by $N\cong\partial\s\times[0,1)$. Since $\Delta\setminus (N* \partial\t) \cong \Delta$, the map $g$ can be defined on the domain $\Delta\setminus(N*\partial\t)$ via the same formula as above. For each point $(x,s,t,y)$ of $N*\partial\t$ we define $g(x,s,t,y):=\alpha H(s,t)$.
\item Now it is easy to see that $${(g|_{\Delta}})^{-1}(0)=\left\{ \left(x, s, {1-s\over s}, \eta(x)\right)\in\Delta: x\in S^{m-n-1}\text{
and }s\in[0.5,1]\right\},$$
where $S^{m-n-1}$ denotes the equatorial subsphere of $\partial\s\cong S^{i-1}$. Because of the identifications $(x,s,1,y)\sim(x',s',1,y)$  and $(x,1,0,y)\sim(x',1,0,y')$ in the join $\Delta=(\partial\s*b_\t) * \partial\t$, we get that the zero set is homeomorphic to $\{(x,s)\in S^{m-n-1}\times[0.5,1]\}/_\sim$ where the equivalence $\sim$ is defined by $$(x,1)\sim(x',1)\text{ for each $x,x'$ and }(x,0.5)\sim (x',0.5)\text{ when }\eta(x)=\eta(x').$$ But this space is homeomorphic to $\Cone(\eta)$ by definition.
\end{itemize}
\end{proof}
\heading{More general incompleteness results?} Theorem~\ref{t:triviality} yields that well group in dimension $m-i$ fails to capture the lack of extendability of $f|_A$ on $X^{(i)}$. Does this lack of extendability imply some robust properties of the zero set?
The answer is yes when $X$ is a triangulable manifold but we will not prove it here.
The lack of extendability implies that (some part of) the zero set of each perturbation projects to at least  $(m-i)$-dimensional subspace of $X$.
More formally, for every $Z\in Z_r(f)$  there is  a simplex $\s\in X$ of dimension $j\leq i$ and its dual cell $\tau$ such that every $j$-disk $B^j$ embedded in $\partial\s * \t$ with $\partial B^j=\partial\s$ intersects $Z$. There is a family of mutually disjoint  disks that are  parameterized by the $(m-j)$-cell $\t$ dual to $\s$ (here we use that $X$ is a manifold). Namely, for each $y$ of the interior of $\t$ we can choose $B_y:=\partial \s * y$. 
This property is not captured by $U_{m-i}(f,r)$ one can construct examples where $H_k(X,B)\cong 0$ for $k<m-i$ and thus each $U_k(f)$ is trivial. (We remark that nontriviality of $H_{m-i}\big(X,B;\pi_ {i-1} (S^{n-1})\big)$ is forced by obstruction theory and Poincar\'e duality when $X$ is a manifold.)
\section{Characterization by homotopy classes}\label{a:characterization}
The proof of 
Proposition~\ref{p:characterization} will utilize certain properties
of compact Hausdorff spaces. All maps are assumed to be continuous,
without explicitly saying it.

We say that a pair of spaces $(Y,Z)$ satisfy \emph{homotopy extension
property with respect to a space $T$} whenever each map $H'\:Y\times\{0\}\cup
Z\times[0,1]\to T$ can be extended to $H\:Y\times[0,1]\to T$.
The map $H'$ as above will be called a \emph{partial
homotopy} of  $H'|_{Y}$ on $Z$. It follows from \cite[Prop. 9.3]{HuBook}
that, once $K$ is compact Hausdorff and $T$ triangulable, every
pair of closed subsets $(Y,Z)$ of $K$ satisfies the homotopy
extension property with respect to $T$. 

In addition, for every two disjoint closed subsets $V$ and $W$
in a compact Hausdorff space $K$ there is a \emph{separating
function} $\chi\:K\to[0,1]$. That means, there is a function
$\chi\:K\to[0,1]$ that is $0$ on $V$ and $1$ on $W$. It is easily
seen that the values $0$ and $1$ above can be replaced by arbitrary
real values $s<t$.

Finally, every homotopy $H\:Y\times[0,1]\to T$ of the form $F(y,t)=F(y,0)$
will be called \emph{stationary}.

We first prove the easier version of Proposition~\ref{p:characterization}
where $\Zf$ is replaced by $Z_r^<(f):=\{g^{-1}(0)\mid g\:K\to
\R^n$ s.t. $\|g-f\|<r\}$. The simpler proof reveals better the main ideas and, in addition, the key step---i.e., the following lemma---is  required also in the full proof.

\begin{lemma}[From perturbations to extensions]\label{l:perturb-ext}
Let $f\:K\to \R^n$ be a map on a compact Hausdorff space $K$
and let $(X,A):=|f|^{-1}([0,r],\{r\})$. Then the families 
\[\label{e:array_1}\tag{A}
\{g^{-1}(0)\mid g\text{ is a strict $r$-perturbation
of } f\}\text{,}
 \]
\[\label{e:array_2}
\{h^{-1}(0)\mid
h\:(X,A)\to (\R^n,\R^n\setminus\{0\})\text{,
}h\sim f|_X\}
\text{ and}\tag{B}\]
\[\label{e:array_3}
\{e^{-1}(0)\mid e\:X\to\R^n \text{
is an extension
of } f|_A\} \tag{C}\] are all equal.\footnote{In (\ref{e:array_2}),
we consider homotopies of maps of pairs.}
Moreover, for every extension $e\:X\to\R^n$ of $f|A$ there is a ``corresponding'' strict $r$-perturbation $g$ with $g^{-1}(0)=e^{-1}(0)$ of the form $g=\chi e$ where  $\chi\:X\to\R^+$ is a positive scalar function.
\end{lemma}

\begin{proof}We will prove that the sequence of inclusions (A)
$\subseteq$ (B) $\subseteq$ (C) $\subseteq$ (A)  holds.

\heading{(A) is a subset of (B):} This inclusion is trivial
since the restriction $g|_X$ of each strict $r$-perturbation $g$ is
homotopic to $f|_X$ as a map of pairs via the straight line homotopy
$F_t=t\,f+(1-t) g$, $t\in [0,1]$.

\heading{(B) is a subset of (C):} We start with a map of
pairs $h$ homotopic to $f|_X$ and want to construct an extension
$e$ of $f|_A$ with the same zero set. To that end, let us choose
a value $\epsilon>0$ such that $\min_{x\in A}|h(x)|\ge 2\epsilon$
and let us define $Y:=|\,h|_X\,|^{-1} [\epsilon,\infty)$. The
partial homotopy of $h$ on $|\,h|_X\,|^{-1}(\epsilon)\cup A$
that is stationary on $|\,h|_X\,|^{-1} (\epsilon)$ and equal
to the given homotopy $h|_A\sim f|_A$ on $A$ can be extended
to $H\:Y\times[0,1]\to\R^n\setminus\{0\}$ by the homotopy extension
property. The homotopy extension property holds because all the
considered maps take values in a triangulable space $\{x\in\R^n\:|x|\in[\epsilon,
M]\}$ for some $M\in\R$. 

The desired extension $e$ can be defined to be equal to $h$ on
$|\,h|_X\,|^{-1}[0,\epsilon]$ and equal to $H(\cdot,1)$ on $Y$.

\heading{(C) is a subset of (A):} We start with an extension
$e\:X\to\R^n$ of $f|_A$ and we want to construct a strict $r$-perturbation
$g$ of $f$ such that $g^{-1}(0)=e^{-1}(0)$. 

The set $U:=\{x\in X\: |e(x)-f(x)|<r/2\}$ is an open neighborhood
of $A$.
Due to the compactness of $|f|^{-1}[0, r]$, there exists $\epsilon\in
(0,r/2)$ such that $|f|^{-1}[r-\epsilon, r]\subseteq U$ (otherwise,
there would exist a sequence $x_n\notin U$ with $|f(x_n)| \to
r$ and a convergent subsequence $x_{j_n}\to x_0$, where $x_0\in
A\subseteq U$, contradicting $x_{j_n}\notin U$).

Let $\chi\:X\to[\epsilon/(2 \|e\|),1]$ be a separating function
for $A$ and  $W:=|f|^{-1} [0,r-\epsilon]$, that is, a continuous
function that is $\epsilon/(2\|e\|)$ on $W$ and $1$ on $A$. The
map $g\:X\to\R^n$ defined by 
$$g(x):=\chi(x)e(x)$$
is a strict $r$-perturbation of $f$. Indeed, for $x\in W$ we have $|g(x)-f(x)|\leq
\epsilon/2 +(r-\epsilon)<r$. Otherwise, $x\in U$ and
then$$|g(x)-f(x)|\leq \chi(x)\underbrace{|e(x)-f(x)|}_{\leq r/2}+
(1-\chi(x))\underbrace{|f(x)|}_{\leq r}<r.$$  
\end{proof}

For the proof of the simpler version of Proposition~\ref{p:characterization} with $\Zf$ replaced by $Z_r^<(f):=\{g^{-1}(0)\mid g\:K\to
\R^n$ s.t. $\|g-f\|<r\}$, the equality $(A)=(B)$
is crucial. 
\begin{proof}[Proof of Proposition~\ref{p:characterization} with $\Zf$ replaced by $Z_r^<(f)$]
It suffices to show that the homotopy class of a map $f\:(X,A)\to
(\R^n,\R^n\setminus\{0\})$ is determined by the homotopy class
of the restriction $f|_A$. That is, we prove that each two maps
of pairs  $f,g\:(X,A)\to (\R^n, \R^n\setminus\{0\})$  satisfying
$f|_A\sim g|_A$ are homotopic. By the homotopy extension property
for the pair
$(X,A)$, the partial homotopy $f|_A\sim g|_A$ of $f$ on $A$ can
be extended to $H\:X\times[0,1]\to \R^n$ with $H(\cdot,0)=f$.

By concatenating $H$ with the straight line homotopy between
$H(\cdot,1)$ and $g$ we obtain the desired homotopy. \end{proof}

The full proof of Proposition~\ref{p:characterization} follows directly from the following analog of the equality $(A)=(B)$ of Lemma~\ref{l:perturb-ext}.
Let us denote by $M$ the mapping cylinder of the inclusion $A \hookrightarrow X$, that is, $M:=X\times\{0\}\cup A\times [0,1]$.
\begin{lemma}\label{l:nonstrict}
$Z_r(f)$ equals to the family
$$
\{h^{-1}(0)\mid h:M\to\R^n\text{ s.t. } h|_{A\times\{1\}}=f|_A
\text{ and } h|_{A\times(0,1]}\text{ avoids } 0\}.$$
\end{lemma}
The maps $h\:M\to\R^n$ as above will be called \emph{homotopy perturbations of $f$.} Clearly the family of zero sets of homotopy perturbations of a $g\sim f$ is equal to the family of zero sets of homotopy perturbations of $f$.  
\begin{proof}[Proof of Lemma~\ref{l:nonstrict}]
First assume that a map $g: K\to\R^n$ satisfies $\|g-f\|\leq r$. Then the map $h\:M\to\R^n$ that is equal to $g$ on $X$ and to the straightline homotopy between $g$ and $f$ on $A\times[0,1]$ is a homotopy perturbation of $f$ with $h^{-1}(0)=g^{-1}(0)$.

Conversely, assume that a homotopy perturbation $h\:M\to\R^n$ of $f$ is given.
We will denote by $h'$ the restriction $h|_X$. Let us define $O_j:=|h'|^{-1}[0,1/j)$.  These sets are open neighborhoods of $h^{-1}(0)$ in $X$, the intersection of all $O_j$ is the zero set $h^{-1}(0)$ and $\bar{O}_{j+1}\subseteq O_j$ (consequently $\bar{O}_{j+1}$ is disjoint from $X\setminus O_j$).
Let as define a partial homotopy $G_1'$ of $h'|_{X\setminus O_2}$ on $(A\setminus O_1)\cup \partial O_2)$ as follows.
We define $G_1'$ to be equal to $h$ on $(A\setminus O_1)\times [0,1]$ and to be the stationary homotopy equal to $h$ on $\partial O_2$.
The partial homotopy $|G_1'|$ is bounded from below and above by positive constants
$m$ and $M$, so we can define the target space of all maps to be the triangulated space $T=\{x\in\R^n:\,m\leq |x|\leq M\}$. 
The homotopy extension property of the pair 
$(X\setminus O_2, (A\setminus O_1)\cup \partial O_2)$ with respect to $T$ implies that $G_1'$ can be extended to a nowhere zero map
$G_1: (X\setminus O_2)\times[0,1]\to T$ such that $G_1(\cdot,0)=h'|_{X\setminus O_2}$.

Inductively, we define homotopies $G_j: (X\setminus O_{j+1})\times[0,1]\to\R^n\setminus \{0\}$ such that
\begin{itemize}
\item $G_j$ equals $G_{j-1}$ on $X\setminus O_{j-1}$,
\item $G_j=H$ on $A\setminus O_j$, 
\item $G_j$ is the stationary homotopy equal to $h|_{\partial O_{j+1}}$ on $\partial O_{j+1}$, and
\item $G_j(0)=h|_{X\setminus O_{j+1}}$.
\end{itemize}
Let $G_j'$ be a partial homotopy of $h'|_{X\setminus O_{j+1}}$ on $(X\setminus O_{j-1})\cup (A\setminus O_{j})\cup (\partial O_{j+1})$
defined by the first three properties of $G_j$ above. 
This is well defined and continuous, because $G_{j-1}$ equals $h$ on $A\setminus O_{j-1}$, and $\partial O_{j+1}$ is disjoint from
the other two parts. 
By the homotopy extension property, there exists
a homotopy $G_j: (X\setminus O_{j+1})\times[0,1]\to\R^n\setminus\{0\}$ satisfying all four properties above.

Let as define continuous maps $g_j: X\to\R^n$ by 
$$
g_j(x)=\begin{cases}
G_j(x,1)\quad\text{ for $x\in X\setminus O_{j+1}$ },\\
h'(x)\quad\text{ for $x\in \bar{O}_{j+1}$.}
\end{cases}
$$
These maps satisfy
\begin{itemize}
\item $g_j=g_{j-1}$ outside $O_{j-1}$,
\item $g_j^{-1}(0)=h^{-1}(0)$,
\item $g_j$ is an extension of $f|_{A\setminus O_j}$.
\end{itemize}

Let $\alpha_j: X\to [1-\frac{1}{j},1]$ be so that $\alpha_j=1$ outside $O_{j-1}$ and $\alpha_j<1$ inside $O_j$.
Define $f_j\:X\to\R^n$ by $f_j:=\alpha_j\,f$. We have that $|f_j|^{-1}(r)\subseteq A\setminus O_j$ and $\|f_j-f\|\to 0$.
The map $\alpha_j g_j$ is an extension of $f_j|_{A\setminus O_j}$ and hence an extension of $f_j|_{|f_j|^{-1}(r)}$, 
so by Lemma~\ref{l:perturb-ext}, some positive scalar multiple
$\beta_j g_j$ of $g_j$ is a strict $r$-perturbation of $f_j$. We will show that $\beta_j\:X\to(0,1]$ may be chosen so that they additionally satisfy
\begin{itemize}
\item $\beta_j=\beta_{j-1}$ outside $O_{j-1}$ (and hence $\beta_j g_j=\beta_{j-1} g_{j-1}$ outside $O_{j-1}$), 
\item $|\beta_j g_j|\leq \frac{1}{j}$ in $\bar{O}_{j}$, and
\item $|\beta_j g_j|\leq |\beta_{j-1}\,g_{j-1}|$ on $X\setminus O_j$.
\end{itemize}

Assume that such $\beta_1,\ldots, \beta_{j-1}$ have been chosen. 
Because $g_{j}=g_{j-1}$ and $f_{j}=f_{j-1}$ outside $O_{j-1}$, we have
$\beta_{j-1} g_{j}=\beta_{j-1} g_{j-1}$ and thus $\beta_{j-1}\,g_j$ is also a strict $r$-perturbation of $f_{j}$ in $X\setminus O_{j-1}$. 
If $\beta_{j}^{'}$ is so that $\beta_{j}'\,g_{j}$ is a global strict $r$-perturbation of $f_{j}$, 
we may define $\beta_{j}''$ to be a positive scalar extension of $\beta_{j}'$ in $\bar{O}_{j}$ and 
of $\beta_{j-1}$ on $X\setminus O_{j-1}$. Then $\beta_j'' g_j$
is a strict $r$-perturbation of $f_{j}$ on $\bar{O}_{j} \cup X\setminus O_{j-1}$.
Furthermore, $\beta_{j}''\,g_{j}$ is a strict $r$-perturbation of $f_{j}$ on some open neighborhood $U$ of $X\setminus O_{j-1}$.
By multiplying $\beta_{j}''$ with a $(0,1]$-valued function that equals $1$ on $X\setminus O_{j-1}$, and is small enough in
${O}_{j-1}\setminus U$, we get a positive function $\beta_j'''$ such that $\beta_{j}'''\leq \beta_{j}'$ in ${O}_{j-1}\setminus U$ 
and that $|\beta_{j}'''\,g_{j}|\leq \frac{1}{j}$ in $\bar{O}_{j}$.  
The resulting $\beta_{j}'''\,g_{j}$ is still a strict $r$-perturbation of $f_{j}$ on $X$ since each $\beta'''_j(x)g_j(x)$ is a strict convex combination of $\beta''_j(x)g_j(x)$ (less than $r$-far from $f_j(x)$) and $0$ (at most $r$-far from $f_j(x)$).
Finally, we multiply $\beta_j'''$ 
by some positive extension $X\to (0,1]$ of the function $\min \{1, |\beta_{j-1} g_{j-1}|/|\beta_{j} g_{j}|\}$ defined on $X\setminus O_j$
to get the desired function $\beta_j$ and then $\beta_j\,g_j$ is a strict $r$-perturbation of $f_j$ 
satisfying all the three properties above.

Let $g(x):=\lim_j \beta_j(x)\,g_j(x)$ for all $x\in X$. This is well defined and continuous. 
If $h(x)\neq 0$, then some neighborhood $U(x)$ of $x$ is contained in $X\setminus O_j$ for $j$ large enough
and for any $y\in U(x)$, $\beta_j(y) g_j(y)=\beta_{j+1}(y)\,g_{j+1}(y)=\ldots$ is stabilized. 
Further, if $h(x)=0$ then for each $j$, some neighborhood of $x$ is contained in $O_j$ and $|\beta_i\,g_i|\leq 1/j$ for each $i>j$ on this neighborhood.
This shows that $g(x)=0$ and $g$ is continuous in $x$. 

By construction, $g^{-1}(0)=h^{-1}(0)$ and the inequality $|\beta_j(x)\,g_j(x) - f_j(x)|<r$ implies that $|g(x)-f(x)|\leq r$ holds for each $x\in X.$

\end{proof}

\begin{remark}[On \emph{computability} of $[ f|_A$\zavr ]
\label{r:comp} 
By \cite[Theorem 1.1]{post}, when $\dim A\leq 2n-4$, the set
$[A,S^{n-1}]$ has a natural structure of an Abelian group with
a distinguished element $[f|_A]$ and its isomorphism type can
be computed. In the case $\dim K=2n-3$ we may not have $\dim
A\leq 2n-4$ but we can still provide a solution as follows. By
an easy inspection of the proofs above, both Lemma~\ref{l:perturb-ext}
and Proposition~\ref{p:characterization} hold with $A:=\partial \big(|f|^{-1}[r,
\infty)\big)$. Once $K$ is a $(2n-3)$-dimensional simplicial
complex and $f$ is simplexwise linear, then $\dim A=2n-4$ for
such choice of $A$.\footnote{However, elsewhere in this paper
we prefer the simpler definition $A:=|f|^{-1}(r)$.} 
\end{remark}

\section{Cap product in \v Cech (co)homology.}
\label{a:Cech}
Let $(X,A)$ be a pair of topological spaces and denote $H_*(X,A)$ ($H^*(X,A)$) the \v Cech (co)homology of $(X,A)$;
we assume a fixed coefficient group $G$ for the homology and the constant sheaf $G$ for the cohomology.
If $\mathcal{U}$ is a covering of $X$ and $\mathcal{V}$ a refinement, 
then $\mathcal{U}_A:=\{U\cap A:\,\,U\in\mathcal{U}\}$ is a covering of $A$ and $\mathcal{V}_A$ is its refinement.
We associate to $\mathcal{U}$ the nerv $N(\mathcal{U})$ of the covering and further define the nerv $N(\mathcal{U}_A)$
to be a simplicial complex whose $q$-simplices are all sets $\{U_0,\ldots,U_q\}\subseteq\mathcal{U}$ containing $q+1$ elements 
such that $(\cap_{j=0}^q U_j)\cap A\neq\emptyset$: this is a subcomplex of $N(\mathcal{U})$.
A map $p: \mathcal{V}\to\mathcal{U}$ that maps each set in $\mathcal{V}$ to a superset in $\mathcal{U}$, 
induces a map $p_A: \mathcal{V}_A\to\mathcal{U}_A$ 
and maps on the nerves $N(\mathcal{V})\to N(\mathcal{U})$, $N(\mathcal{V}_A)\to N(\mathcal{U}_A)$, which are 
simplicial maps between simplicial complexes. 
Any other choice $p': \mathcal{V}\to\mathcal{U}$ yields a homotopic map $N(\mathcal{V})\to N(\mathcal{U})$, so a subcovering induces 
well-defined maps $p_*$ resp. $p^*$ between  simplicial (co)homologies of the nerves 
$p_*: \mathcal{H}_*(X,A,\mathcal{V})\to\mathcal{H}_*(X,A,\mathcal{U})$ and 
$p^*: \mathcal{H}^*(X,A,\mathcal{U})\to\mathcal{H}^*(X,A,\mathcal{V})$. 
%
The \v Cech (co)homology is defined by $H_*(X,A):=\varprojlim_{\mathcal{U}} \mathcal{H}_*(X,A,\mathcal{U})$ and
$H^*(X,A):=\varinjlim_{\mathcal{U}} \mathcal{H}^*(X,A,\mathcal{U})$.

Let $\phi\in H^*(X,A)$ and $\beta\in H_*(X,A\cup B)$ for some $A,B\subseteq X$. 
The direct limit can be defined as a disjoint union of all $\mathcal{H}^*(X,A,\mathcal{U})$ factored by the relation $\lambda-p^*(\lambda)$
and the inverse limit can be defined as $\{\mu\in\prod \mathcal{H}_*(X,A\cup B, \mathcal{U}):\,\,p_*(\mu_{\mathcal{V}})=\mu_\mathcal{U}\}$.
Therefore, $\beta$ can be represented as a net 
$(\beta_{\mathcal{U}})_\mathcal{U}$ so that $\beta_\mathcal{U}=p_*\beta_\mathcal{V}$ for any refinement $\mathcal{V}$ of any covering 
$\mathcal{U}$, and $\phi$ can be represented by an element $\phi_{\tilde{\mathcal{U}}}$ contained in $\mathcal{H}^*(X,A,\tilde{\mathcal{U}})$ 
for some fine enough $\tilde{\mathcal{U}}$.

For each such $\tilde{\mathcal{U}}$, we define $\alpha_{\tilde{\mathcal{U}}}:=\phi_{\tilde{\mathcal{U}}} \frown \beta_{\tilde{\mathcal{U}}}\in\mathcal{H}_*(X,B,\tilde{\mathcal{U}})$
by means of simplicial cap product. If $\tilde{\mathcal{V}}$ is a refinement of $\tilde{\mathcal{U}}$, then the naturality of 
simplicial cap product---namely, the relation
$\phi_{\tilde{\mathcal{U}}}\frown p_*(\beta_{\tilde{\mathcal{V}}})=p_*(p^*(\phi_{\tilde{\mathcal{U}}})\frown \beta_{\tilde{\mathcal{V}}})$---implies that $\alpha_{\tilde{\mathcal{U}}}=p_* \alpha_{\tilde{\mathcal{V}}}$.
Therefore, we may consistently define $\alpha_\mathcal{U}$ for any $\mathcal{U}$ to be equal to $p_* \alpha_{\mathcal{V}}$
where $\mathcal{V}$ is a fine enough refinement of $\mathcal{U}$ such that  $\alpha_\mathcal{V}$ is already defined. This construction yields
an element $\alpha\in H_*(X,B)$ and we define the \v Cech cap product $\phi\frown\beta$ to be equal to $\alpha$.

Let $f: X\to X'$, $A,B\subseteq X$, $A', B'\subseteq X'$, $f(A)\subseteq A'$ and $f(B)\subseteq B'$.
We will show the naturality of $\frown$, that is, the relation $\phi'\frown f_*(\beta)=f_*(f^*(\phi')\frown \beta)$ 
for $\phi'\in H^*(X',A')$ and $\beta\in H_*(X, A\cup B)$.

Let $\mathcal{U}$ be a covering of $X$ and $\mathcal{U}'$ a covering of $X'$ so that for any $U\in\mathcal{U}$ there exists
$U'\in\mathcal{U}'$ such that $f(U)\subseteq U'$. Given $\mathcal{U'}$ and $f$, then any covering of $X$ 
admits a refinement $\mathcal{U}$ with this property.
Such map $f$ and a choice of the set $U'$ for any $U$ induces a simplicial map $f_\sharp: N(\mathcal{U})\to N(\mathcal{U'})$;
other assignment of the supersets $U'$ of $f(U)$ yields a homotopic map.
If $\mathcal{V}$ resp. $\mathcal{V}'$ is a refinement of $\mathcal{U}$ resp. $\mathcal{U}'$ such that
$f$ maps each $V\in\mathcal{V}$ into some $V'\in\mathcal{V}'$, then we have a square of simplicial maps between simplicial complexes
\begin{equation}
\hskip 3.5cm
\begin{diagram}
\label{e:ref_square}
N(\mathcal{U})   & \rTo^{f_\sharp} & N(\mathcal{U}') \\
\uTo^{p}         & \ruTo^{f_\sharp}       & \uTo_{p'}       \\ 
N(\mathcal{V})   & \rTo^{f_\sharp} & N(\mathcal{V}')
\end{diagram}
\end{equation}
that commute up to homotopy and induce commuting maps on the level of their simplicial (co)homology.
Thus, if $\beta\in H_*(X,B)$ is represented by $(\beta_\mathcal{U})_\mathcal{U}$, then $f_*v$ can be defined to by
the net that assigns to each $\mathcal{U}'$ the $f_*$-image of $\beta_{\mathcal{U}}\in\mathcal{H}(X, A\cup B, \mathcal{U})$ 
where $\mathcal{U}$ is fine enough so that $f$ maps elements of $\mathcal{U}$ to $\mathcal{U}'$. 
This is well defined, because if $\mathcal{V}$ is a refinement of $\mathcal{U}$, then the commutativity of the induced maps
in the upper left triangle of (\ref{e:ref_square}) implies $f_* \beta_\mathcal{V}=f_* p_* \beta_\mathcal{V}=f_* \beta_\mathcal{U}$;
further, any $\tilde{\mathcal{U}}$ such that $f$ maps its element to elements of $\mathcal{U'}$ has a common refinement with $\mathcal{U}$.
The commutativity of the induced maps in the lower right triangle of (\ref{e:ref_square}) implies that 
$p_* (f_* \beta)_{\mathcal{V}'}=(f_* \beta)_{\mathcal{U}'}$, hence $f_*\beta$ is a well defined element of $H_*(X,B)$.

Let $\phi'\in H^*(X',A')$ be represented by some $\phi_{\tilde{\mathcal{U'}}}'\in \mathcal{H}^*(X',A',\tilde{\mathcal{U}'})$.
By construction, $\phi'\frown f_* \beta$ is represented by the net that assigns to any refinement $\tilde{\mathcal{V}'}$
of $\tilde{\mathcal{U}'}$ the cap product of 
\begin{equation}
\label{e:upper_side}
\phi_{\tilde{\mathcal{V}'}}'\frown f_*(\beta_{\mathcal{U}})\in\mathcal{H}_*(X',B', \tilde{\mathcal{V}'})
\end{equation} where $\mathcal{U}$
is fine enough so that $f$ maps each $U\in\mathcal{U}$ to some $\tilde{V}'\in\tilde{\mathcal{V}'}$. It remains to show that
$f_*(f^* \phi'\frown \beta)$ is represented by the same object.
Let $\tilde{\mathcal{V}'}$ be a refinement of $\tilde{\mathcal{U}'}$, $\mathcal{U}$ be as in (\ref{e:upper_side}) and
let $\mathcal{V}$ be a refinement of $\mathcal{U}$ fine enough so that $f$ maps each of its sets to some element of $\tilde{\mathcal{V}'}$.
The term $f_*(\beta_\mathcal{U})$ in (\ref{e:upper_side}) equals $f_* (\beta_\mathcal{V})$ and
$f^*(\phi')$ can be represented by $f^* (\phi_{\tilde{\mathcal{V}'}}')\in\mathcal{H}^*(X,A,\mathcal{V})$.
It follows that $f_* (f^* \phi' \frown \beta)$ can be represented by a net that
assigns to $\tilde{\mathcal{V}'}$ the element $f_* (f^* (\phi_{\tilde{\mathcal{V}'}}')\frown \beta_{\mathcal{V}})\in\mathcal{H}_*(X',B',\tilde{\mathcal{V}'})$
which equals (\ref{e:upper_side}) by the naturality of simplicial cap product.
\end{document}